\documentclass[10pt,reqno]{amsart}
\usepackage{mathtools,amssymb,color,eucal,mathrsfs,bm}
\usepackage{mathptmx}
\usepackage{enumitem}
\setitemize{leftmargin=*, itemsep=3pt, topsep=3pt}   
\setenumerate{leftmargin=*, itemsep=3pt, topsep=3pt, 
              label=\textup{(\arabic*)}}             

\usepackage[text={150mm,240mm},centering]{geometry}
\usepackage[noadjust]{cite}

\usepackage{array}
\newcolumntype{L}{>{$\displaystyle}l<{$}}
\newcolumntype{C}{>{$}c<{$}}

\usepackage{multirow}

\usepackage{graphicx}
\usepackage{tikz}
\usetikzlibrary{arrows}
\usetikzlibrary{decorations.pathreplacing,calligraphy}

\tikzset{
	wt/.style={circle, draw=black, fill=black!40!white, inner sep=2pt, outer sep=0pt, minimum size=5pt}, 
	axiscolour/.style={blue}, 
}

\usepackage[colorlinks=true,citecolor=red,linkcolor=blue]{hyperref} 

\usepackage{setspace}

\usepackage[colorlinks=true,citecolor=red,linkcolor=blue
]{hyperref} 

\allowdisplaybreaks
\numberwithin{equation}{section}

\usepackage[capitalise,noabbrev]{cleveref} 

\newtheorem{thm}{Theorem}[section]
\newtheorem{cor}[thm]{Corollary}
\newtheorem{lemma}[thm]{Lemma}

\newtheorem{prop}[thm]{Proposition}

\newtheorem{Properties}[thm]{Properties}

\newtheorem{theorem}[thm]{Theorem}
\newtheorem{corollary}[thm]{Corollary}
\newtheorem{proposition}[thm]{Proposition}
\newtheorem{conjecture}[thm]{Conjecture}

\newtheorem{assum}{Assumption}

\theoremstyle{definition}
\newtheorem{defin}[thm]{Definition}
\newtheorem{remark}[thm]{Remark}

\newtheorem{definition}[thm]{Definition}

\Crefname{thm}{Theorem}{Theorems}
\Crefname{prop}{Proposition}{Propositions}
\Crefname{lemma}{Lemma}{Lemmas}
\Crefname{cor}{Corollary}{Corollaries}
\Crefname{defn}{Definition}{Definitions}
\Crefname{prob}{Problem}{Problems}
\Crefname{conj}{Conjecture}{Conjectures}
\Crefname{assum}{Assumption}{Assumptions}
\Crefname{result}{Result}{Results}

\definecolor{darkgreen}{rgb}{0.1, 0.8, 0.1}

\newcommand{\alg}[1]{\mathfrak{#1}}  
\newcommand{\Mod}[1]{\mathcal{#1}}   
\newcommand{\VOA}[1]{\mathsf{#1}}    
\newcommand{\fld}[1]{\mathbb{#1}}    
\newcommand{\lat}[1]{\mathsf{#1}}    

\newcommand{\ZZ}{\fld{Z}}

\newcommand{\rlat}{\lat{Q}} 

\DeclarePairedDelimiter{\sqbrac}{\lbrack}{\rbrack} 
\DeclarePairedDelimiter{\set}{\lbrace}{\rbrace}

\DeclarePairedDelimiterX{\comm}[2]{\lbrack}{\rbrack}{#1 , #2}  
\DeclarePairedDelimiterX{\acomm}[2]{\lbrace}{\rbrace}{#1 , #2} 
\DeclarePairedDelimiterX{\inner}[2]{\lparen}{\rparen}{#1 , #2} 
\DeclarePairedDelimiterX{\super}[2]{\lparen}{\rparen}{#1 \delimsize\vert \mathopen{} #2} 

\newcommand{\SLA}[2]{\alg{#1}_{#2}}                       

\newcommand{\sfaut}{\sigma}                   
\newcommand{\sfmod}[2]{\sfaut^{#1}(#2)}       

\newcommand{\savoa}[2]{\VOA{L}_{#1}(#2)}  
\newcommand{\uavoa}[2]{\VOA{V}^{#1}(#2)}  

\newcommand{\slvoa}[1]{\savoa{#1}{\SLA{sl}{2}}}

\newcommand{\slirr}[1]{\Mod{L}_{#1}}              
\newcommand{\sldis}[1]{\Mod{D}_{#1}}              
\newcommand{\slindrel}[1]{\Mod{E}_{#1}}           
\newcommand{\slrel}[2]{\slindrel{#1;\Delta_{#2}}} 
\newcommand{\slproj}[1]{\Mod{P}_{#1}}             

\DeclareMathOperator{\chmap}{ch}
\newcommand{\Gr}[1]{\sqbrac[\big]{#1}}                               
\newcommand{\ch}[1]{\chmap \Gr{#1}}                                  

\DeclareMathOperator{\Ext}{Ext}
\DeclareMathOperator{\Hom}{Hom}

\newcommand{\ra}{\rightarrow}
\newcommand{\lra}{\longrightarrow}

\newcommand{\ses}[3]{0 \ra #1 \ra #2 \ra #3 \ra 0}      
\newcommand{\dses}[3]{0 \lra #1 \lra #2 \lra #3 \lra 0} 

\newcommand{\voa}{vertex operator algebra}

\newcommand{\g}{\mathfrak{g}}
\newcommand{\h}{\mathfrak{h}}
\newcommand{\p}{\mathfrak{p}}
\newcommand{\gu}{\mathfrak{u}}
\newcommand{\gl}{\mathfrak{l}}
\newcommand{\gb}{\mathfrak{b}}
\newcommand{\ag}{\widehat{\g}}
\newcommand{\ah}{\widehat{\h}}

\newcommand{\C}{\mathbb{C}}
\newcommand{\lam}{\lambda}

\newcommand{\cR}{\mathcal{R}}
\newcommand{\cD}{{\mf{sl}_2\on{-wtmod}_{\Omega}}}
\newcommand{\cE}{\mathcal{E}}

\newcommand{\ad}{\mathrm {ad}}
\newcommand{\Sing}{\mathrm{Sing}}
\newcommand{\im}{\mathrm{im}}

\newcommand{\mc}{\mathcal}
\newcommand{\mf}{\mathfrak}
\newcommand{\on}{\operatorname}
\newcommand{\bigWt}[1]{\VOA V^{#1}(\g)\on{-wtmod}}
\newcommand{\bigWtpg}[1]{\VOA V^{#1}(\g)\on{-wtmod}_{\geq 0}}
\newcommand{\bigWtO}[1]{\VOA V^{#1}(\g)\on{-wtmod}_{\mc{O}}}
\newcommand{\bigWtKL}[1]{\VOA V^{#1}(\g)\on{-wtmod}_{\on{KL}}}
\newcommand{\sWt}[1]{\VOA L_{#1}(\g)\on{-wtmod}}
\newcommand{\sWtpg}[1]{\VOA L_{#1}(\g)\on{-wtmod}_{\geq 0}}
\newcommand{\sWtO}[1]{\VOA L_{#1}(\g)\on{-wtmod}_{\mc{O}}}
\newcommand{\sWtKL}[1]{\VOA L_{#1}(\g)\on{-wtmod}_{\on{KL}}}
\newcommand{\sWtsl}[1]{\VOA L_{#1}(\mf{sl}_2)\on{-wtmod}}
\newcommand{\sWtpgsl}[1]{\VOA L_{#1}(\mf{sl}_2)\on{-wtmod}_{\geq 0}}

\begin{document}

\title{Weight representations of affine Kac-Moody algebras and small quantum groups}

\author{Tomoyuki Arakawa}
\address{Research Institute for Mathematical Sciences,
Kyoto University, Kyoto 606-8502, Japan 
	}
\email{arakawa@kurims.kyoto-u.ac.jp}

\author{Thomas Creutzig}
\address{
	Department of Mathematical and Statistical Sciences \\
	University of Alberta\\
	Edmonton, Canada T6G2G1
}
\email{creutzig@ualberta.ca}
\thanks{T.C.~is supported by a NSERC Discovery Grant}

\author{Kazuya Kawasetsu}
\address{Priority Organization for Innovation and Excellence, Kumamoto University, Kumamoto 860-8555, Japan
}
\email{kawasetsu@kumamoto-u.ac.jp}
\thanks{K.K. is partially supported by MEXT Japan ``Leading Initiative for Excellent Young Researchers (LEADER)'',
JSPS Kakenhi Grant numbers 19KK0065, 21K13775 and 21H04993.
}



\begin{abstract}
 We
study the weight modules over affine Kac-Moody algebras
from the view point of vertex algebras,
and
 determine the abelian category of weight modules for the simple affine vertex algebra $\slvoa{k}$ at any
non-integral admissible level $k$.
In particular,
we show that the principal block of the category of weight modules over admissible $\slvoa{k}$
is equivalent to that  of the corresponding (unrolled) small quantum group.

\end{abstract}

\maketitle

\onehalfspacing 

\section{Introduction}

In this paper we reveal a surprising connection between
the weight modules over affine Kac-Moody algebras and 
representations of small quantum groups.

 Although the classification of simple weight modules of an affine Kac-Moody algebra has been achieved in \cite{FT-weight},
very little is known about its structure as an abelian category.
Probably one of the reasons is that 
a priori it is not clear at all
whether the category of weight representations  over affine Kac-Moody algebras
possess nice structures as in the case of the Bernstein-Gelfand-Gelfand category $\mathcal{O}$ \cite{Kac-book}.
In fact it is known that the category of weight modules with finite-dimensional weight spaces over a finite dimensional complex simple Lie algebra
do {\em not} have  enough projectives \cite{Maz10}.

The key idea of this paper is that 
 we 
study the weight modules over affine Kac-Moody algebras
from the view point of vertex algebras.
Namely,
we study
the category of weight modules over the {\em simple affine vertex algebra},
which forms a full subcategory of that of the  corresponding affine Kac-Moody algebra.
When the simple affine vertex algebra is relatively nice,
we can expect that the corresponding
 category of weight modules has a nice structure.

Indeed,  in the case where the simple affine vertex algebra is integrable as a representation of the affine Kac-Moody algebra,
it is known \cite{FreZhu12} that the category of weight modules   is semisimple, and coincides with the category of integrable representations 
of the Kac-Moody algebra of the same level as that of the simple affine vertex algebra.
However,  in view of the representation theory of weight representations this category is not interesting in the sense that the only simple weight modules belonging to this category are highest weight representations. 

Next nicest simple affine vertex algebras are admissible affine vertex algebras,
which are {\em quasi-lisse} in the sense of \cite{AraKaw18} as proved in \cite{Ara09b}.
By \cite{AdaVer95,A12-2},  the category of modules over an admissible affine vertex algebra that belong to the category $\mathcal{O}$
is semisimple and consists of Kac-Wakimoto admissible representations \cite{KacWak88,KacWak89}.
However, unlike integrable affine vertex algebras, 
admissible affine vertex algebras do have non-highest weight weight representations,
and the classification of all simple weight representations with lower bounded conformal weights of an admissible affine vertex algebra has been achieved in  \cite{KawRid21}.

Let $\sWt{k}$ be the category of smooth weight representations with finite-dimensional weight spaces 
 of the simple affine vertex algebra $\savoa{k}{\mathfrak{g}}$.
Our main conjecture is the following.
\begin{conjecture}\label{main-conjecture} 
Let $\savoa{k}{\mathfrak{g}}$ be an admissible affine vertex 
algebra at a 
principal admissible level $k$.
\begin{enumerate}
\item The category  $\sWt{k}$  of 
has enough projectives.
More precisely,
every simple object has a finite length projective cover.
\item
 The principal block of $\sWt{k}$  
is equivalent to that of the (unrolled) small quantum group $\mathfrak{u}_t(\check{\g})$ associated with the Langlands dual $\check{\mathfrak{g}}$
of $\g$ with $t=e^{\pi \sqrt{-1}(k+h^{\vee})}$,
where $h^{\vee}$ is the dual Coxeter number of $\g$.
\end{enumerate}

\end{conjecture}
Beyond $\mathfrak{sl}_2$, the case of $\mathfrak{sl}_3$ at admissible level has been studied in more detail in \cite{Kawasetsu:2021qls, ACG}. In the instance of the level $-\frac{3}{2}$ the results have been compared to the category of weight modules of the unrolled small quantum group $\mathfrak{u}^H_{\sqrt{-1}}(\mathfrak{sl}_3)$
in \cite{Creutzig:2021seo} and are consistent with our conjecture.

In this paper we prove Conjecture \ref{main-conjecture}
for $\mathfrak{g}=\mathfrak{sl}_2$. More precisely, 
we show that a block of $\sWt{k}$  is equivalent to either the {\em typical} category $\mathcal C^{\text{typ}}$
or the  {\em atypical} category $\mathcal C^{\text{atyp}}$,
where $\mathcal C^{\text{typ}}$ is the category of finite-dimensional vector spaces that is semisimple  with a single simple object
and $\mathcal C^{\text{atyp}}$ is the category that is discusses in detail in section \ref{sec:atyp}. 
The category $\mathcal C^{\text{atyp}}$  has inequivalent simple objects $L_n$ for $n\in \mathbb Z$ and the projective cover and injective hull $P_n$ of $L_n$ has Loewy diagram 
%
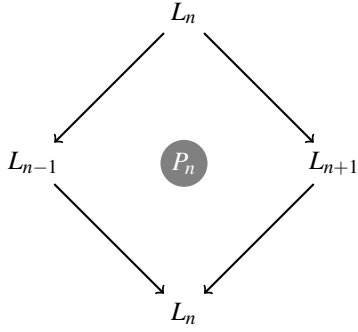
\begin{center}
\begin{tikzpicture}[scale=1]
\node (top) at (0,2) [] {$L_n$};
\node (left) at (-2,0) [] {$L_{n-1}$};
\node (right) at (2,0) [] {$L_{n+1}$};
\node (bottom) at (0,-2) [] {$L_n$};
\draw[->, thick] (top) -- (left);
\draw[->, thick] (top) -- (right);
\draw[->, thick] (left) -- (bottom);
\draw[->, thick] (right) -- (bottom);
\node (label) at (0,0) [circle, inner sep=2pt, color=white, fill=black!50!] {$P_n$};
\end{tikzpicture}
\end{center}
Next consider the set $\{ (r, s) \in \mathbb Z^2 \, | \, 1 \leq r \leq u-1, \ 1 \leq s \leq v-1\}$ with the equivalence relation $(r, s) \sim (r, s)$ and $(r, s) \sim (u-r, v-s)$ and let $I_{u, v}$ denote the set of equivalence classes.
\begin{theorem}
Let $k = -2+\frac{u}{v}$ for $(u, v) =1$, $u, v \in \mathbb Z_{\geq 2}$ be a non-integrable admissible level for $\g=\mathfrak{sl}_2$. 
Let $\lambda_{r, s} = r -1 -\frac{u}{v}s$.
The category $\sWtsl{k}$ has the block decomposition
\begin{equation}
\sWtsl{k} \cong \Bigg( \bigoplus_{r=1}^{u-1} \bigoplus_{s=0}^{v-1} \ \mathcal C^{\text{atyp}} \Bigg) \ \oplus  \
 \Bigg( \bigoplus_{(r, s) \in I_{u, v}}  \ \bigoplus_{\ell \in \mathbb Z}  \ \bigoplus_{\substack{\lambda \in \mathbb C/2\mathbb Z \\ \lambda \neq \lambda_{r, s}, \lambda_{u-r, v-s}\, \text{mod}\, 2 } } \  \mathcal C^{\text{typ}} \Bigg)
\end{equation}
and $\mathcal C^{\text{atyp}}$ is equivalent to the principal block of the unrolled small quantum group $\mathfrak{u}_t( \mathfrak{sl}_2)$  with $t=e^{\pi \sqrt{-1}(k+2)}$,

\end{theorem}
Note that the unrolled quantum group  $u_q^{H}(\mf{sl}_2)$ of $\g=\mathfrak{sl}_2$ at $2r$-th root of unity has a similar decomposition, namely 
\[
u_q^{H}(\mf{sl}_2)\on{-mod} \cong \bigoplus_{\alpha \in \mathbb C \setminus \mathbb Z \cup r\mathbb Z } \ \mathcal C^{\text{typ}}  \  \oplus \  \bigoplus_{i=0}^{r-2}\  \mathcal C^{\text{atyp}}.
\]

In both cases the principal block is atypical. 
We note that they are also equivalent to the principal block of the restricted category $\mc{O}$ of $\widehat{\mf{sl}_2}$ at the critical level
\cite{AraFie12}.

{We also note that
there is a unique object in an atypical block
$\mathcal C^{\text{atyp}}\subset \sWt{k}$ that lies in  the category $\mathcal{O}$ of $\widehat{\g}$.
For instance, in the principal block the simple affine vertex algebra
$\savoa{k}{\mathfrak{g}}$ is such a unique object.
Therefore,
the connection  between the quantum group and the affine Kac-Moody algebra studied in this article
is quite different from the ones by  Kazhdan and Lusztig \cite{KazLus93,KazLus94}
and Lusztig \cite{Lus80}, see also \cite{AraFie12}.}

\subsection{Representation theory of affine vertex algebras at admissible level}

Our motivation to study this problem is a better understanding of the representation theory of affine vertex algebras beyond integrable levels. 
Let $\g$ be a simple Lie algebra, $h$ the Coxeter number and $h^\vee$ the dual one. Let $\savoa{k}{\mathfrak{g}}$ be the simple affine vertex algebra of $\g$ at level $k$. 
The level $k$ is said to be admissible if $k = - h^\vee + \frac{u}{v}$ for coprime positive integers $u, v$ and 
$u \geq h^\vee$ if $v$ is coprime to the lacing number of $\g$ and $u \geq h$ otherwise. 
We are interested in the category $\sWt{k}$  of finitely generated smooth weight modules that are $\savoa{k}{\mathfrak{g}}$-modules and that have finite dimensional weight spaces. 
This category has a variety of interesting subcategories. Firstly the category of relaxed-highest weight modules, $\sWtpg{k}$, that is modules whose conformal weights are in addition lower bounded; and secondly the category of ordinary modules,
$\sWtKL{k}$,
that is finitely generated lower bounded weight modules with finite dimensional conformal weight spaces.
We recall the state of the art of the understanding of these categories and how our work contributes. 

Firstly, ordinary modules are very well understood. They form a finite semisimple braided tensor category \cite{A12-2, CreBra17}. It is also believed that ordinary modules are rigid, this is currently proven in type ACD and E and some additional special cases \cite{CreFus, CGL-osp, CKL-para}.
In types ACD and E fusion rules coincide with the ones of associated principal $W$-algebras \cite{CreFus, CGL-osp} and similar results also hold for ordinary modules of the affine vertex superalgebra associated to $\mathfrak{osp}_{1|2n}$ \cite{CGL-osp}.

Relaxed-highest weight modules were first studied in the case of $\g = \mathfrak{sl}_2$. Already more than 25 years ago, Adamovic and Milas classified simple modules in $\sWtsl{k}_{\geq0}$ \cite{AdaVer95}.  The general case has been treated much more recently \cite{KawRid21}.
The work of Adamovic and Milas falls into the early days of logarithmic conformal field theory and WZW models at fractional level serve as prototypical examples. The adjective logarithmic refers to the appearance of logarithmic singularities in correlation functions, a feature that follows from non semi-simple Virasoro algebra zero-mode action \cite{CreLog13}. 
A consistent full conformal field theory is constructed out of a representation category of modules of the underlying vertex algebra that is closed under fusion and under modular transformations. Relaxed highest weight modules don't close under fusion as exemplified first by Gaberdiel \cite{GabFus01} and then also by Ridout \cite{RidFus10} and Adamovic \cite{AdaRea17}, but also their images under spectral flow appear. Spectral flow refers to modules twisted by automorphisms associated to translations of the affine Weyl group.  These modules also appear in modular transformations \cite{CreMod12, CreMod13} and the conjectural Verlinde formula for fusion rules. 
The conclusion is that the category of relaxed-highest weight modules is too small and it is better to pass to a larger category, the category $\sWtsl{k}$  of finitely generated smooth weight modules that have finite dimensional weight spaces. 

There are two levels, $k = -1/2$ and $k=-4/3$ for which $\VOA L_k(\SLA{sl}{2})$ allows particular nice free field realizations \cite{Adamovic:2004zi, CreCos13}. In particular the case $k=-1/2$ has been studied extensively by David Ridout \cite{RidSL208, RidSL210, RidFus10}. Our main result has already been obtained recently in these two special cases \cite{CMY5}.

By \cite{AKR23}, we know that every simple module in $\sWt{k}$  is a spectral flow image of a simple module in $\sWtpg{k}$, see Theorem \ref{spec-flow}. Our contribution is then a complete description of $\sWtsl{k}$ and it is exactly as expected \cite[Conjecture 2]{Creutzig:2018ogj}. We proceed as follows:

\subsubsection{Outline of proof} 
${}$

\noindent {\bf Step 1} Extensions in $\sWtpg{k}$. 

Firstly, we need to understand the abelian category of relaxed-highest weight modules. We prove that this problem can be reduced to the top level, that is to the representation theory of $\g$. More precisely we have \Cref{prop:relext}. This result holds for any simple Lie algebra $\g$.
\begin{theorem}[\Cref{prop:relext}]
Let $\g$ be a finite-dimensional simple Lie algebra and $k$ be admissible.
Let $M,N$ be finite-length, almost simple, relaxed-highest weight $\savoa k\g$-modules with the lowest conformal weights $h_M$ and $h_N$.
Then the following hold. 
\begin{enumerate} \item If $h_M= h_N$, then $\Ext^1(M,N)\cong \Ext^1(M_{top},N_{top})$.
\item  If $h_M\neq h_N$, then $\Ext^1(M,N)=0$.
\end{enumerate}\end{theorem}

\noindent {\bf Step 2} Ext$^1$ in $\sWtsl{k}$.

Next, we derive various criteria so that $\Ext^1(M, N)=0$.  There is a trivial conformal weight criterion and another one if certain weight spaces vanish, see Theorem \ref{thm:extaaa}. We also need that certain extensions exist. These are obtained from the explicit construction of modules \cite{AdaRea17, Kawrel19}. In particular Adamovic's construction of logarithmic modules is very important. Since a few such extensions aren't covered by Adamovic's construction  we develop a new technique that also covers the missing cases, see appendix \ref{appendix}. We expect that this construction will turn out to be useful in the future when studying categories corresponding to higher rank $\mathfrak g$. 

As a consequence we get the following characterization: Let $L_0$ be a simple module in $\sWtsl{k}$ that is not projective, then the block $\mathcal C$ 
containing $L_0$ has the properties:
\begin{Properties}\label{properties}
 $\mathcal C$ is a category whose simple inequivalent objects are $L_n$ for $n \in \mathbb Z$ and 
\begin{enumerate}
\item 
$
\Ext^1(L_n, L_m) = \begin{cases} \mathbb C & \quad  \text{if} \ |n-m| =1 \\ 0 &\quad \text{else;} \end{cases}
$
\item Let
$ \ses{L_n}{E^\pm_n}{L_{n \pm1}} \ \in  \ \Ext^1( L_{n \pm 1}, L_n)$
{be a non-zero element.}
Then 
$\Ext^1(E^+_{n+1}, L_n) = 0 = \Ext^1(L_n, E^+_{n-2})\ $ and $\ \Ext^1(E^-_{n-1}, L_n) = 0 = \Ext^1(L_n, E^-_{n+2})$;
\item 
$\Ext^1(E^+_{n-1}, L_n) = 0 = \Ext^1(L_n, E^+_n)\ $ and  $\ \Ext^1(E^-_{n+1}, L_n) = 0 = \Ext^1(L_n, E^-_n)$;
\item 
$\Ext^1(E^+_n, E^+_{n+1}) \neq 0 \ \text{and} \ \Ext^1(E^-_n, E^-_{n-1}) \neq 0$.
\end{enumerate}
\end{Properties}
The atypical blocks of the category of weight modules of $u_q^H(\mathfrak{sl}_2)$ have the exact same properties, see section \ref{sec:unrolledsl2}.

\noindent {\bf Step 3}  Study the atypical block.

Finally, we determine consequences of the above properties and it turns out that they completely characterize the atypical block.  For this we look at many long exact sequences in homology and observe that all extensions are completely determined, see Theorem \ref{main}. 
As a summary we have
\begin{theorem} \textup{(Corollary \ref{maincor})}
As an abelian category the atypical block is completely determined by Properties \ref{properties} and in particular equivalent to the atypical block of $u^H_q(\SLA{sl}{2})$.
\end{theorem}

\subsubsection{Consequences}

One can hope that our proof strategy generalizes well. The first step is already done for all $\g$, and the second one should generalize well, however combinatorics will be more involved. But the idea is to find enough vanishing of $\Ext^1$ together with constructing enough extensions so that one hopefully can proof that these properties already determine the block uniquely. One of the biggest obstacles is the construction of logarithmic modules. At least for $\g = \mathfrak{sl}_3$ there is progress \cite{ACG} and with the method of Appendix \ref{appendix} more extensions should be constructable.

Eventually one would like to have and understand tensor category structure on $\sWtsl{k}$. A priori this looks like an impossible task as any existence theorem only works for ordinary modules. However $\sWtsl{k}$ enjoys Kazama-Suzuki duality \cite{Kazama:1988qp} and in particular there is a block-wise equivalence of categories of $\sWtsl{k}$ and of the $N=2$ super conformal algebra at central charge $3k/(k+2)$ \cite{Feigin:1997ha, CGNR,CLRW}. This means that the simple $N=2$ super conformal algebra at central charge $3k/(k+2)$ for admissible $k$ has the same typical and atypical blocks as $\sWtsl{k}$.  The main result of \cite{CY} establishes existence of vertex tensor category provided that ordinary modules are of finite length and $C_1$-cofinite, see \cite[Theorem 3.3.5]{CY} for the precise statement. We establish finite length and the $C_1$-cofiniteness is proven in \cite{C}. Kazama-Suzuki duality is good enough that tensor category of $\sWtsl{k}$
follows from \cite{CreTen17, CMY}.

Next one might ask if one has actual tensor equivalences \`a la Kazhdan-Lusztig of $\sWtsl{k}$ and some (quasi)-Hopf algebra. There has been quite some recent progress in this direction \cite{CLR1, CLR2} and in particular the technology of \cite{CLR2} hints that there is a connection to a quotient category of $u^H_q(\mathfrak{sl}_{2|1})$ for $q = e^{\pi \sqrt{-1}(k+2)}$, see also \cite{CNS} for an analogue at levels of the form $k = -2 + \frac{1}{p}$ and in particular a proof in the special case of $p=1$.

\subsection{Outline}

We start in section \ref{QG} with a short description of categories of modules of unrolled quantum groups. In section \ref{section:generalitiles} we introduce various categories of modules and explain the state of the art for $\mathfrak{sl}_2$. \Cref{sec:catlb} studies first extensions in the category of lower-bounded modules and then in \Cref{sec:catwt} discusses vanishing of extension criteria on the category of weight modules. These are applied to the case of $\mathfrak{sl}_2$. In the final section \ref{sec:atyp} we then  prove that the previously determined extensions and vanishing extensions completely characterize the atypical blocks of $\sWtsl{k}$.

\subsection*{Acknowledgements} We thank David Ridout and Drazen Adamovic for useful discussions.

\section{Unrolled small quantum groups}\label{QG}

We use \cite{Ru} as reference, see also \cite{MR4194285}.
Let $\g$ be a simple Lie algebra of rank $n$, $\h$ a Cartan subalgebra with corresponding positive simple roots $\alpha_1, \dots, \alpha_n$. Let $( \ \ ,  \ \ )$ be the Killing form normalized such that short roots have norm two. For a simple root $\alpha_i$, set $d_i = (\alpha_i, \alpha_i)/2$ and $a_{i, j} = (\alpha_i, \alpha_j)/d_i$. Let $q$ be a primitive $\ell$-th root of unity for $\ell \in \mathbb Z_{\geq 3}$ and set $q_i := q^{d_i}$. 
Let $P$ be the weight lattice of $\g$ and $Q$ its root lattice. 
\begin{definition}
The unrolled  quantum group $U_q^H(\g)$  of $\g$ associated to a lattice $L$, $Q \subset L \subset P$, is the $\mathbb C$-algebra with generators $K_\gamma, X^\pm_i, H_i$ with $i= 1, \dots, n$ and $\gamma \in L$ and relations
\begin{equation}
\begin{split}
K_0 &= 1, \qquad K_\gamma K_\mu = K_{\gamma +\mu}, \qquad K_\gamma X_i^\pm K_{-\gamma} = q^{\pm (\gamma, \alpha_i)} X^\pm_i \\ 
[H_i, H_j]&= 0, \qquad [H_i, K_\gamma] = 0, \qquad [H_i, X^\pm_j] = a_{ij}X^\pm_j \\
[X^+_i, X^-_j] &= \delta_{i, j} \frac{K_{\alpha_j} - K_{-\alpha_j}}{q_j - q_j^{-1}} \\
&\sum_{k=0}^{1- a_{ij}} (-1)^k \binom{1-a_{ij}}{k}_{q_i} (X^\pm_{i})^kX^\pm_{j} (X^{\pm}_{ i})^{1-a_{ij}-k}=0 \qquad \text{if $ i \not = j$}.
\end{split}
\end{equation} 
\end{definition} 
The unrolled  quantum group
$U_q^H(\g)$ is a Hopf algebra with coproduct
\[
\Delta(K_\gamma) = K_\gamma \otimes K_\gamma,  \quad \Delta(X^+_i) = 1 \otimes X^+_i + X^+_i \otimes K_{\alpha_i},
\quad \Delta(X^-_i) = K_{-\alpha_i} \otimes X^-_i + X^-_i \otimes K_{\alpha_i}, \Delta(H_i) = 1 \otimes H_i + H_i \otimes 1,
\]
counit
\[
\epsilon(K_\gamma) = 1, \qquad \epsilon(X^\pm_i) = 0 = \epsilon(H_i),
\]
and antipode
\[
S(K_\gamma) = K_{-\gamma}, \quad S(X^+_i) = - X_i K_{-\alpha_i}, \quad S(X^-_i) = -K_{\alpha_i}X^i_i, \quad S(H_i) = - H_i.
\]
\begin{definition}
The small (or restricted) unrolled quantum group $u_q^H(\g)$ of $\g$ associated to $L$ is the quotient of $U_q^H(\g)$  by the Hopf ideal generated by $\{ (X^\pm_i)^{r_i}\}_{i = 1}^n$ where $r_i = \frac{r}{\text{gcd}(d_i, r)}$.
\end{definition}
Let $\mathcal C$ be the category of finite dimensional weight modules of $u_q^H(\g)$. Weight refers to a semisimple action of the Cartan subalgebra, such that $K_\gamma = \prod\limits_{i=1}^n q_i^{k_iH_i}$ for $\gamma = \sum\limits_{i=1}^n k_i H_i$. 
$\mathcal C$ has the following properties
\begin{enumerate}
\item Invertible objects are parameterized by coweights of $\g$ \cite[Remark 4.7]{Ru};
\item $\mathcal C$ is graded  by a maximal torus of $\g$  and the generic graded blocks are semisimple \cite[Proposition 4.8]{Ru};
\item $\mathcal C$ is ribbon \cite[Corollary 4.9]{Ru};
\item $\mathcal C$ has trivial M\"uger center \cite[Proposition 4.10]{Ru}
\item $\mathcal C$ is unimodular and projective covers are also injective hulls \cite[Corollary 4.15]{Ru}
\end{enumerate}

The small quantum group $u_q(\g)$  is a subquotient of $u_q^H(\g)$. Namely the quotient of the subalgebra that excludes the $H_i$  by the relations $K_\mu = 1$ for $\mu \in Q \cap \ell P$. 
This means a module of $u_q^H(\g)$ descends to a module of $u_q(\g)$ if and only if $K_\mu$ acts as the identity on it for any $\mu \in Q \cap \ell P$. This property holds in particular for objects in the principal block. Moreover non-isomorphic objects in $u_q^H(\g)$ can become isomorphic as objects for the small quantum group and in particular the principal block of the small quantum group only has finitely many inequivalent simple objects. Unlike the unrolled quantum group, the small quantum group is not braidable \cite{kondo} and a quasi-Hopf modification is needed \cite{Creutzig:2017khq}. This is obtained via uprolling/modularization \cite{Creutzig:2020jxj, Gainutdinov:2018pni,negron}.

\subsection{$u_q^H(\mathfrak{sl}_2)$}\label{sec:unrolledsl2}

In the case of $\g = \mathfrak{sl}_2$ we can be more explicit \cite{MR3320217}. For each complex number $\alpha$ one has a typical module $V_\alpha$. This module is simple and projective unless $\alpha \in \mathbb Z \setminus r\mathbb Z$. 
In the latter atypical case these modules are composed of simple $i+1$-dimensional modules $S_i \otimes \mathbb C_{\ell r}^H$ for $i = 0, \dots, r-2$ and $\ell \in \mathbb Z$.  Namely
\[
\ses{S_{i} \otimes \mathbb C_{\ell r}^H}{V_{r-i-1+\ell r}}{S_{r-i-2}\otimes \mathbb C_{(\ell +1)r}^H} 
\]
and the projective cover and injective hull $P_i \otimes \mathbb C_{\ell r}^H$ of $S_i \otimes \mathbb C_{\ell r}^H$ has Loewy diagram
\begin{center}
\begin{tikzpicture}[scale=1]
\node (top) at (0,2) [] {$S_i \otimes \mathbb C_{\ell r}^H$};
\node (left) at (-2,0) [] {$S_{r-i-2} \otimes \mathbb C_{(\ell-1) r}^H$};
\node (right) at (2,0) [] {$S_{r-i-2} \otimes \mathbb C_{(\ell +1) r}^H$};
\node (bottom) at (0,-2) [] {$S_i \otimes \mathbb C_{\ell r}^H$};
\draw[->, thick] (top) -- (left);
\draw[->, thick] (top) -- (right);
\draw[->, thick] (left) -- (bottom);
\draw[->, thick] (right) -- (bottom);
\node (label) at (0,0) [circle, inner sep=2pt, color=white, fill=black!50!] {$P_i \otimes \mathbb C_{\ell r}^H$};
\end{tikzpicture}
\end{center}
Let $\mathcal C_i$ the block of $S_i \otimes \mathbb C_{0}^H$ and $\mathcal C_\alpha$ the block of $V_\alpha$ for $\alpha \notin \mathbb Z \setminus r\mathbb Z$. Each $\mathcal C_\alpha$ is semisimple with a single simple object. The block decomposition of
$u_q^H(\mathfrak{sl}_2)\on{-mod}$ 
is 
\[
u_q^H(\mathfrak{sl}_2)\on{-mod}
= \bigoplus_{\alpha \in \mathbb C \setminus \mathbb Z \cup r\mathbb Z } \mathcal C_\alpha \  \oplus \  \bigoplus_{i=0}^{r-2} \mathcal C_i.
\] 
We now describe $\mathcal C_i$. 
Set 
\[ 
L_n := \begin{cases} S_i \otimes \mathbb C_{n r}^H & n \ \text{even} \\ S_{r-i-2} \otimes \mathbb C_{n r}^H & n \ \text{odd} \end{cases}
\]
these are the simple objects in $\mathcal C_i$. 
Let $P_n$ be the projective cover of $L_n$. 
We list some properties. 
\begin{enumerate}
\item The projective cover tells us that
\begin{equation}\label{propQG1}
\Ext^1(L_n, L_m) = \begin{cases} \mathbb C & \quad  \text{if} \ |n-m| =1 \\ 0 &\quad \text{else} \end{cases}
\end{equation}
\item Let
\begin{equation}
 \ses{L_n}{E^\pm_n}{L_{n \pm1}} \qquad \in  \ \Ext^1( L_{n \pm 1}, L_n)
\end{equation}
be a non-zero element. 
They satisfy the short exact sequences
\begin{equation}
\label{exttypqg}
\ses{E^\pm_n}{P_n}{E^\pm_{n \mp 1}}
\end{equation}
and from the corresponding long exact sequences in homology together with projectivity and injectivity of the $P_n$ one immediately gets
\begin{equation}\label{propQG2}
\Ext^1(E^+_{n+1}, L_n) = 0 = \Ext^1(L_n, E^+_{n-2}), \qquad \Ext^1(E^-_{n-1}, L_n) = 0 = \Ext^1(L_n, E^-_{n+2})
\end{equation}
\item 
as well as 
\begin{equation}\label{propQG3}
\Ext^1(E^+_{n-1}, L_n) = 0 = \Ext^1(L_n, E^+_n), \qquad \Ext^1(E^-_{n+1}, L_n) = 0 = \Ext^1(L_n, E^-_n)
\end{equation}
\item By \eqref{exttypqg}
\begin{equation}\label{propQG4}
\Ext^1(E^+_n, E^+_{n+1}) \neq 0 \qquad \text{and} \qquad \Ext^1(E^-_n, E^-_{n-1}) \neq 0.
\end{equation}
\end{enumerate}
We will prove later that any two categories satisfying above four properties, i.e. \eqref{propQG1},  \eqref{propQG2},  \eqref{propQG3} and  \eqref{propQG4}, are abelian equivalent. 

\section{Generalities on  affine Lie algebra modules}\label{section:generalitiles}

\subsection{Preliminaries}

Let $\g$ be a finite dimensional simple  Lie superalgebra, or $\g = \mathfrak{gl}_{n|n}$, with invariant, non-degenerate, supersymmetric bilinear form $\kappa$. Let $\h$ be a Cartan subalgebra of dimension $\ell$ and $\Pi$ a corresponding set of positive simple roots. $\kappa$ induces a bilinear form on $\h^*$ and we normalize $\kappa$ such that long roots have norm two if $\g$ is a Lie algebra. In the Lie superalgebra case we choose a simple even subalgebra of $\g$ and normalize $\kappa$ such that the long roots of this subalgebra have norm two. $\h$ and $\h^*$ are identified via $\kappa$ and we denote this by $\nu : \h \rightarrow \h^*,  \ \ h \mapsto \kappa(h,  \, \cdot \, )$. 
The coroot $\alpha^\vee \in \h$ of $\alpha$ is determined by $\kappa( \alpha^\vee, h) = \alpha(h)$ for all $h \in \h$. 
The lattices spanned by coroots and roots are the coroot lattice $Q^\vee$ and the root lattice $Q$. The fundamental weights are in $\h^*$ and they are the natural duals of the simple coroots and the fundamental coweights are in $\h$ and are the natural duals of the simple roots. They span the weight and coweight lattices $P$ and $P^\vee$. 

The affinization of $\g$ is (we write $\g[t, t^{-1}]$ for $\g \otimes_{\C} \C[t, t^{-1}]$)
\[
\ag = \g[t, t^{-1}] \oplus  \C \, K \oplus \C \, d 
\]
with $K$ central and non-vanishing commutation relations
\[
[x \otimes t^{m}, y \otimes t^{n}] = [x, y] \otimes t^{m+n} +  m\, \delta_{m+n, 0}\, K\, \kappa(x, y). \qquad [d, x \otimes t^n] = n x \otimes t^n.
\]
The Cartan subalgebra $\h$ extends to $\ah = \h \oplus  \C \, K \oplus \C \, d$ and $\kappa$ extends to $\ah$ via $\kappa(K, \h) = \kappa(d, \h) = \kappa( K, K) = \kappa(d, d) = 0$ and $\kappa(K, d) = 1$. $\nu$ then extends accordingly to $\ah$. 
Any coweight $\lambda \in P^\vee$ acts on $\h^*$ via affine Weyl translation $t_\lambda$ 
\[
t_\lambda(\alpha) = \alpha + \alpha(K)\nu(\lambda) - \left( \alpha(\lambda) + \frac{1}{2} \kappa(\lambda, \lambda) \alpha(K)\right) \nu(K)
\]
and these induce automorphisms on the affine Lie superalgebra. These are called spectral flow automorphisms $\sigma^\lambda$ and for $\lambda \in P^\vee$ they act as 
\begin{equation}\label{spectralflow}
\begin{split}
\sigma^\lambda(e^\alpha \otimes t^n) &= e^\alpha \otimes t^{n - \alpha(\lambda)}, \qquad e^\alpha \in  \g_\alpha = \{ x \in  \g | [h, x] = \alpha(h)x \ \forall \ \h \in \h \} \\
\sigma^\lambda(h \otimes t^n) &= h \otimes t^n - \delta_{n, 0} \kappa(\lambda, h)K , \qquad h \in  \h \\
\sigma^\lambda(K) &= K \\
\sigma^\lambda(d) &= d + \lambda \otimes t^0 -\frac{1}{2} \kappa(\lambda, \lambda) K
\end{split}
\end{equation}

\subsection{The category of smooth weight modules}

Let $M$ be a $\ag$-module. $M$ is called smooth if for every $m \in M$ and every $x \in \g$ there exists a sufficiently large $n \in \mathbb Z$, such that $(x \otimes t^N).m = 0$ for all $N \geq n$. We can construct a twisted module associated to $\lambda \in P^\vee$. Namely, let $\sigma^\lambda(M)$ be isomorphic to $M$ as a vector space, with isomorphism given by $m \mapsto \sigma^*_\lambda(m)$ and twisted action defined as
\[
X.\sigma^*_\lambda(m) = \sigma^*_\lambda\left( \sigma^{-\lambda}(X).m\right), \qquad X \in U(\ag), \ m \in M.
\]
The spectral flow twist has the two obvious properties
\begin{enumerate}
\item If $M$ is smooth then so is $\sigma^\lambda(M)$.
\item If $ 0 \rightarrow  M_1 \rightarrow M_2 \rightarrow M_3 \rightarrow 0$ is a non-split exact sequence of $\ag$-modules, then the same is true for
$ 0 \rightarrow  \sigma^\lambda(M_1) \rightarrow \sigma^\lambda(M_2) \rightarrow \sigma^\lambda(M_3) \rightarrow 0$.
\end{enumerate}
We call an $\ag$-module $M$ to be of level $k \in\C$ if $K$ acts on $M$ by multiplication with $k$. 
The affine vertex superalgebra $\uavoa k\g$ is generated by fields $x(z)$ for $x\in \g$ with OPEs
\[
x(z)y(w) = \frac{k\kappa(x, y)}{(z-w)^2} + \frac{[x, y](w)}{(z - w)}
\]
and as a $\ag$-module $\uavoa k\g$ is the vacuum Verma module at level $k$
\[
\uavoa k\g \cong  U(\ag) \otimes_{U(\g[t] \oplus \C K \oplus d \C)} \C_k, 
\]
where $\C_k$ is the one-dimensional $\g[t] \oplus \C K \oplus  \C d$-module on which $K$ acts by multiplication with $k$ and $\g[t] \oplus  \C d$ acts as zero.
 We denote the simple quotient of $\uavoa k\g$  by $\savoa{k}{\mathfrak{g}}$.  The Sugawara vector is denoted by $L(z) = \sum_{n \in \mathbb Z}L_n z^{-n-2}$ and generalized eigenvalues of $L_0$ are called conformal weights.

We list some basic types of modules
\begin{enumerate}
\item
A module $M$ over $\g$ is called a {\em weight} module
if $\h$ acts semisimply on $M$:
$$
M=\bigoplus_{\lam\in\h^*}M_\lam.
$$
\item
Let $\gb$ be a Borel subalgebra containing $\h$.
The Verma module and irreducible highest weight module of weight $\lam\in\h^*$ are denoted by
$M(\lam)$ and $L(\lam)$, respectively.
\item
Let $\p$ be a parabolic subalgebra of $\g$ with the nilradical $\gu$ and Levi factor $\gl$.
Let $N$ be an irreducible weight $\gl$-module.
It is trivially a $\p$-module by setting $\gu.N=0$
and therefore, we have the $\g$-module
$$
M_\p(N)=U(\g)\otimes_{U(\p)}N.
$$
Let $L_\p(N)$ be the irreducible quotient of $M_\p(N)$, 
which is called the parabolically induced module induced from $N$.


\item Let $M$ be a $\ag$-module.
\ As in the finite case, $M$ is called a weight module
if $\ah$ acts semisimply on $M$.

\item Let $M$ be a weight module over $\ag$.
A weight vector $v\in M$ is called a {\em relaxed-highest weight vector}
if $(\g[t]t).v=0$.
\item The module $M$ is called {\em relaxed-highest weight} if $M$
is generated by a single relaxed-highest weight vector $v\in M$.

\item Let $N$ be a weight module over $\g$ and $k$ a non-critical complex number.
It is naturally a $\g[t]\oplus \C K$-module by setting
$$
(\g[t]t).v=0,\quad K.v=kv\quad(v\in N).
$$
The {\em relaxed Verma module} generated by $N$ is the $\ag$-module
$$
M_k(N)=U(\ag)\otimes_{U(\g[t]\oplus\C K)}N,
$$
which is a relaxed-highest weight module if $N$ is generated by a
single weight vector.

\item The {\em almost irreducible quotient} of a lower bounded module $M$
is the quotient module $M/I$, where $I$ is the sum of all submodules
which intersect $M_{top}$ trivially.
A lower bounded module $M$ is called {\em almost irreducible}
if $M$ is generated by $M_{top}$ and there are
 no non-zero submodules of $M$
which intersect $M_{top}$ trivially.
Note that an almost simple module $M$ is irreducible if and only if 
$M_{top}$ is irreducible.
The almost irreducible quotient of $M_k(N)$ is denoted by $L_k(N)$,
which is a relaxed-highest weight module if $N$ is generated by a
single weight vector.
The module $L_k(N)$ is irreducible if and only if $N$ is irreducible.

\end{enumerate}

\begin{defin}Let $k$ be a complex number.
We set 
\begin{enumerate}
\item $\bigWt{k}$
 the category of finitely generated smooth weight modules at level $k$ with finite-dimensional weight spaces, that is objects $\mathcal M$ are finitely generated smooth $\widehat{\mathfrak{g}}$-modules 
 at level $k$ such $\mathfrak{h}$ acts semisimply and so $\mathcal M$ is graded by conformal weight and $\mathfrak{h}$, that is 
\[
\mathcal M = \bigoplus_{\lambda, \Delta} \mathcal M_{\lambda, \Delta}
\]
and
$\text{dim}   \mathcal M_{\lambda, \Delta} < \infty$ for any $(\lambda, \Delta)$. 
(Note that by Proposition 3.5 in \cite{AKR23}, if $k\neq0$, for any module $\mathcal M$ in $\bigWt{k}$, we have the following existence of lower bounds of conformal weights: for each $\lambda$ there exists $h_\lambda$, such that $\mc{M}_{\lambda, \Delta} = 0$ for $\text{Re}(\Delta) < \text{Re}(h_\lambda)$.)
\item $\bigWtpg{k}$ the full subcategory
of $\bigWt{k}$ consisting of objects that are 
finitely generated lower bounded weight modules, that is $\mathcal M$ in $\bigWtpg{k}$, if it is in $\bigWt{k}$ and if conformal weight is lower bounded. 

\item $\bigWtO{k}$ the full subcategory
of $\bigWtpg{k}$ consisting of objects that lies in the category $\mc{O}$ of $\g$.

\item $\bigWtKL{k}$ the full subcategory
of $\bigWtO{k}$ consisting of objects  with finite dimensional conformal weight spaces.

\item $\sWt{k}$ the full subcategory of $\bigWt{k}$ consisting of objects that are $\savoa{k}{\mathfrak{g}}$-modules.

\item $\sWtpg{k}=\sWt{k}\cap \bigWtpg{k}$.

\item $\sWtO{k}=\sWt{k}\cap \bigWtO{k}$.

\item $\sWtKL{k}=\sWt{k}\cap \bigWtKL{k}$.

\end{enumerate}
\end{defin}

\begin{remark}\label{rem:centralchar}
Let $\g$ be a simple Lie algebra and $k$ an admissible level, then any module for Zhu's algebra of $\savoa k\g$ admits a central character \cite[Remark 3.5]{Arakawa:2019ear}. Zhu's algebra is a quotient of $U(\g)$ and the top level of a module $M$ in $\sWtpg{k}$ is a module of Zhu's algebra and hence is in particular a $\g$-module that admits a central character. 
\end{remark}

Write $\sigma^*_{\ell}$ and $\sigma^\ell$ for $\sigma^*_{\ell \omega}$ and $\sigma^{\ell\omega}$ and $\omega$ the fundamental weight of $\SLA{sl}{2}$. 
\begin{lemma} 
Let $k\neq -2,0$. Let $\mathcal M$ be a  module in $\VOA V^{k}(\mf{sl}_2)\on{-wtmod}$.  Assume that $\mathcal M = \langle x \rangle$ is generated by a single homogeneous element. Then $\mathcal M$ has a finite filtration 
\[
0 = M_0 \subset M_1 \subset M_2 \subset \dots \subset M_n = \mathcal M, 
\]
such that for each $i = 1, \dots, n$ there exist an integer $m_i$ such that $\sigma^{m_i}(M_i/M_{i+1})$ is in $\VOA V^{k}(\mf{sl}_2)\on{-wtmod}_{\geq 0}$.
\end{lemma}
\begin{proof}
The generator $x$ is homogeneous meaning it has some weight, say $\mu$ and some (generalized) conformal weight, say $h_x$. Let $N$ be the dimension of 
\[
 \bigoplus_{ \Delta \leq h_x} \mathcal M_{\mu, \Delta}.
\]

We construct $M_1$.

Firstly let us recall equation (2.11) of \cite{CreMod13}, namely that if $v$ has weight $\mu$ and conformal weight $h$, then $\sigma^*_{\ell}(v)$ has weight $\mu + \ell k$ and conformal weight $h + \frac{\ell \mu}{2} + \frac{\ell^2k}{4}$.

Consider the lowest conformal weight $h_\mu$ appearing in 
$\bigoplus_{\Delta}\mathcal M_{\mu,\Delta}$, which implies $h_n v = 0$ for all $n>0$ and all $v \in \mathcal M_{\mu, h_\mu}$. Set $A_r = \{ v\in  \mathcal M_{\mu, h_\mu} | e_n v = 0 \ \quad \text{for all} \ n > r\}$. 
Then we have the chain 
$$
\cdots \supset A_1\supset A_0\supset A_{-1}\supset \cdots
$$ 
of subspaces of $\mathcal M_{\mu,h_\mu}$.
Since $\mathcal M$ is smooth  there exists for every $v \in \mathcal M$ an
 $m \in \mathbb Z$, such that $e_n v = 0$ for all $n >  m$. 
 As in addition $\mathcal M_{\mu,h_\mu}$ is finite-dimensional, the descending chain vanishes at some point and let us consider the number $m\in\ZZ$ which satisfies
 $A_m\neq 0$ and $A_r=0$ for all $r<m$.
Note that every non-zero element $v$ of $A_m$ satisfies $e_mv \neq0$ by the definition of $m$.
Consider the element $X = f_{-m}e_m$ and $v\in A$ and let $n>m$, then using that $h_{n-m}v=0$ it is easy to see that $e_nXv =0$, i.e. the action of $X$ leaves $A$ invariant. We fix $v$ to be an eigenvector of $X$ in $A$, that is $e_nv=0$ for all $n>m$, $e_mv\neq0$ and $Xv$ is a scalar multiple of $v$. 

 Let $w = \sigma^*_m(v)$. Then $\sigma^{-m}(f_0 e_0) = f_{-m}e_m = X$ and so
 $w$ is an eigenvector for $f_0e_0$. Let $a$ be the eigenvalue.
 Also   $e_0 w = e_0 \sigma_m ^*(v) = \sigma_m^*(\sigma^{-m}(e_0) v) =
\sigma_m^*(e_m v) \neq  0$ and $e_n w = e_n \sigma_m ^*(v) = \sigma_m^*(\sigma^{-m}(e_n) v) =
\sigma_m^*(e_{n+m} v) =  0$ for $n>0$. Since $\sigma^{-m}(h_n) = h_n$ for $n\neq 0$ we also have $h_n w = 0$ for $n > 0$
Let $\nu= \mu + mk$ be the $h_0$-eigenvalue of $w$. Since $ \mathcal M_{\mu, h_\mu -n} = 0$ for $n>0$ also  $\sigma^m(\mathcal M)_{\nu, h_w-n} = 0$ for $n>0$ by \eqref{spectralflow}, where $h_w$ is the conformal weight of $w$.

 There are two cases

{\bf Case 1}: $a \neq 0$

We claim that $f_n w = 0$ for all $n>0$ if $f_0e_0 w = a w$ for some $a\neq 0$. Clearly it is enough to show that $f_1 w = 0$. 
We also note that if $f_0 w = 0 $, then also $[f_0, h_1] w = 0$ and hence $f_1 w =0$.

Consider $e_0w$. Firstly, $e_1e_0w = e_0e_1 w = 0$. Next $h_1 e_0 w = [h_1, e_0] w =  0$ since $e_1w =0$. 
Finally $f_1 e_0 w \in \sigma_m^*(\mathcal M)_{\nu, h_\nu-1} = 0$ and so $f_1e_0 w = 0$. 
Assume that $f_0e_0 w = a w$ for some $a\neq 0$. Then $f_1 a w = f_1 f_0 e_0 w = f_0 f_1 e_0 w = 0$, proving our claim. 
It follows that $w$ generates a module in $\VOA V^{k}(\mf{sl}_2)\on{-wtmod}_{\geq 0}$. Then we set $M_1=\langle v\rangle\subset \mathcal M$.
 
{\bf Case 2}: $a = 0$ 
 
 Lastly if $a = 0$, then $e_0w$  is annihilated by $e_n, h_n, f_n$ for $n>0$ and also by $f_0$. In particular $e_0w$  generates a module in $\VOA V^{k}(\mf{sl}_2)\on{-wtmod}_{\geq 0}$. 
 We set $M_1=\langle e_m v\rangle \subset \mathcal M$.
The second case can only happen if $e_0w$ is a lowest-weight vector for $\SLA{sl}{2}$. The weight of $e_0w$ is $\lambda = \mu +2 + m k$ and the conformal weight is 
$h_\mu+ \frac{\mu m}{2}+ \frac{m^2k}{4}$, but the conformal weight of a lowest-weight vector is
$\Delta = \frac{\lambda(\lambda-2)}{4(k+2)}$. Hence case 2 can only happen if 
\[
\frac{\lambda(\lambda-2)}{4(k+2)} =  h_\mu + \frac{\mu m}{2}+ \frac{m^2k}{4}, \qquad \text{with} \qquad \lambda = \mu +2 + m k.
\]

We repeat this procedure inductively for $\mathcal M/M_i$ to construt $M_{i+1}/M_i$. $M_{i+1}$ is then the pullback, $M_{i+1} = \mathcal M \times_{(\mathcal M/M_i)} (M_{i+1}/M_i)$.
Also set $N_i$ to be the dimension of 
\[
 \bigoplus_{\Delta\leq h_x} (\mathcal M/M_i)_{\mu, \Delta}.
\]
This procedure terminates if the image of $x$ in $M_i$ is in the top level of $M_{i, \mu}$. 
If in the construction case 1 happens, then $N_{i+1} < N_i$ and if case 2 happens then $N_{i+1} \leq N_i$. Since case 2 happens at most a finite number of times this procedure terminates after at finitely many iterations. 
\end{proof}

\begin{corollary}\label{spec-flow}
Let $k\neq -2,0$. Let $M$ be a simple module in $\VOA V^{k}(\mf{sl}_2)\on{-wtmod}$, then there exists a $m \in \mathbb Z$, such that $\sigma^*_m(M)$ is in $\VOA V^{k}(\mf{sl}_2)\on{-wtmod}_{\geq 0}$.
\end{corollary}
This result was shown in Section 3 of \cite{AKR23} by another method (see Remark 3.2 and Theorem 3.11 in \cite{AKR23}). Another Corollary is:

 \begin{corollary}\label{cor:FL}
Let $k\neq -2,0$. If $\VOA V^{k}(\mf{sl}_2)\on{-wtmod}_{\geq 0}$ resp. $\VOA L_{k}(\mf{sl}_2)\on{-wtmod}_{\geq 0}$ is of finite length then so is $\VOA V^k(\mf{sl}_2)\on{-wtmod}$ resp. $\VOA L_k(\mf{sl}_2)\on{-wtmod}$.
 \end{corollary}

\subsection{$\sWtpgsl{k}$ at admissible level}\label{subsecsl2basics}

Let $\g = \SLA{sl}{2}$. We will soon specialize to admissible level. Remark \ref{rem:centralchar} tells us that a module of Zhu's algebra at admissible level is necessarily a $\SLA{sl}{2}$-module with semisimple action of the Casimir. We thus set $\cD$ to be the category of $\SLA{sl}{2}$-modules that have semisimple action of the chosen Cartan subalgebra and in addition also semisimple action of the Casimir. 

 Let $\omega$ be the fundamental weight and $\lambda \omega$ a weight. The irreducible highest-weight (lowest-weight) representation of this highest-weight (lowest-weight) is denoted by $D^+_{\lambda}$ ($D^-_\lambda$). The dense modules of weight $\lambda\omega$ and Casimir eigenvalue $\Delta$ will be denoted by $R_{\lambda, \Delta}$ for generic $(\lambda, \Delta)$ as they are generically simple. (See \Cref{sec:loc} for the definition of dense modules.) These modules come in a continuous (coherent) family parameterized by $\h/Q^\vee$. At non-generic position there are then two modules $R^\pm_{\lambda, \Delta}$, where the superscript indicates that the dense module has a highest ($+$) or lowest ($-$) weight module as submodule. For example for $\lambda$ not dominant integral and if $\Delta = \Delta_\lambda$ is the Casimir eigenvalue on $D^+_\lambda$, then $R^+_{\lambda, \Delta_\lambda}$ is the non-trivial extension
\[
0 \rightarrow D^+_\lambda \rightarrow R^+_{\lambda, \Delta_\lambda} \rightarrow D^-_{\lambda + 2} \rightarrow 0, \qquad \in \ \Ext^1_\cD(D^-_{\lambda + 2}, D^+_\lambda).
\]
It is easy to see that this is the only non-trivial extension of $D^-_{\lambda + 2}$ by $D^+_\lambda$ in 
$\cD$.
See, section 3.1 of \cite{Kawrel19} for more details on relaxed modules of $\SLA{sl}{2}$. 

The notation for modules of the affine \voa{} used in the literature is $\sldis{\lambda}^\pm:=L_k(D^\pm_\lambda)$ 
and if $D^+_\lambda$ is also of lowest-weight then one also uses $\slirr{\lambda+1}$ for  $\sldis{\lambda}^+$. The modules induced from dense modules are denoted by 
$\slrel{\lambda}{}:=L_k(R_{\lambda, \Delta}) $ and similarly 
we set $\mc{E}^{\pm}_{\lambda,\Delta}=L_k(R^\pm_{\lambda, \Delta})$. 

We specialize to admissible level, that is
\begin{equation}
	k+2 = t = \frac{u}{v}, \qquad u \in \ZZ_{\ge 2}, \quad v \in \ZZ_{\ge 1}, \quad \gcd \set{u,v} = 1.
\end{equation}
and introduce
\begin{equation} \label{eq:DefLambda}
		\lambda_{r,s} = r - 1 - ts.
	\end{equation}
	and 
	\begin{equation} \label{eq:DefDelta}
	\Delta_{r,s} = \frac{(r-ts)^2-1}{4t} = \frac{(vr-us)^2-v^2}{4uv}
\end{equation}
and the short-hand notation 
\[
\sldis{r, s}^\pm := \sldis{\pm \lambda_{r, s}}^\pm, \qquad \slirr{r} := \slirr{\lambda_{r, 0} +1},  \qquad \slindrel{r, s}^\pm:= \slrel{\lambda_{r, s}}{{r, s}}^\pm.
\]
\begin{thm} \label{rhwsimples} \textup{(Adamovi\'c-Milas \cite{AdaVer95}, see also \cite{RidRel15})}
	Let $k = -2 + \frac{u}{v}$ be an admissible level. Then, 
	the simple objects in $\sWtpgsl{k}$ 
	are exhausted, up to isomorphism, by the following list:
	\begin{itemize}
		\item The $\slirr{r}$, for $r = 1,\dots,u-1$;
		\item The $\sldis{r,s}^\pm$, for $r = 1,\dots,u-1$ and $s = 1,\dots,v-1$;
		\item The $\slrel{\lambda}{r,s}$, for $r = 1,\dots,u-1$, $s = 1,\dots,v-1$ and $\lambda \in \alg{h}^*$ with $\lambda \neq \lambda_{r,s}, \lambda_{u-r,v-s} \pmod{\rlat}$.
		\end{itemize}
	Apart from the identifications $\slrel{\lambda}{r,s} = \slrel{\lambda}{u-r,v-s}$ and $\slrel{\lambda}{r,s} = \slrel{\mu}{r,s}$, if $\lambda = \mu \pmod{\rlat}$, the modules in this list are all mutually non-isomorphic.
\end{thm}
Therefore, for any $\lam+\rlat\in \h^*/\rlat$, there are only finitely many irreducible modules up to isomorphisms in $\sWtsl{k}_{\geq0}$  having weights in $\lam+\rlat$.
It implies that every block in $\sWtsl{k}_{\geq0}$ contains only finitely many irreducible modules up to isomorphisms.  
Thus, we see by the finite-dimensionality of weight spaces that 
 $\sWtsl{k}_{\geq0}$ is of finite length.
By \Cref{cor:FL}, we have the following corollary.

\begin{corollary}
 $\sWtsl{k}$ is of finite length.
\end{corollary}

There are some identifications
\begin{prop} \label{twistRules} \textup{\cite{CreMod13}}
\begin{equation} \label{eq:SFIdentifications}
	\begin{split}
		\sfmod{}{\slirr{r}} &\cong \sldis{u-r,v-1}^+,\quad \sfmod{-1}{\slirr{r}} \cong \sldis{u-r,v-1}^-,\qquad  r=1,\dots,u-1, \\		\sfmod{-1}{\sldis{r,s}^+} &\cong \sldis{u-r,v-1-s}^-,
	\qquad\qquad\qquad\qquad\qquad\ \ \,
		r=1,\dots,u-1,\quad s=1,\dots,v-2. \\
	\end{split}
\end{equation}
\end{prop}
Finally the following extensions are known
\begin{prop} \label{existelle}
\cite{AdaRea17, Kawrel19}
There exist indecomposable modules $\sfmod{\ell}{\slindrel{r,s}^+}$ fitting in the non-split exact sequences
	\begin{equation} \label{es:DED}
	\begin{split}
		\dses{\sfmod{\ell}{\sldis{r,s}^+}}{\sfmod{\ell}{\slindrel{r,s}^+}}{\sfmod{\ell}{\sldis{u-r,v-s}^-}}, \\
		\dses{\sfmod{\ell}{\sldis{r,s}^-}}{\sfmod{\ell}{\slindrel{r,s}^-}}{\sfmod{\ell}{\sldis{u-r,v-s}^+}},
		\end{split}
	\end{equation}
for $\ell \in\ZZ, r = 1, \dots,  u-1, s = 1, \dots, v-1$. 	
\end{prop}
Note that these extensions are related via conjugation as the conjugate module of $\sfmod{}{\sldis{r,s}^+}$ is $\sfmod{}{\sldis{r,s}^-}$ 
and in fact only the first type of extension was constructed in \cite{AdaRea17, Kawrel19} as the other one follows by conjugation. The same applies in the next proposition. 
\begin{prop}\label{prop:log} \cite{AdaRea17}  and Appendix \ref{appendix}.
There exist indecomposable modules $\sfmod{\ell}{\slproj{r,s}^+}$ fitting in the non-split exact sequences
\begin{align*}
&\dses {\sfmod{\ell}{\slindrel{r,s}^+}} {\sfmod{\ell}{\slproj{r, s}^+}}    {\sfmod{\ell+1}{\slindrel{r,s+1}^+}},\\
&\dses {\sfmod{\ell}{\slindrel{r,v-1}^+}} {\sfmod{\ell}{\slproj{r, v-1}^+}}    {\sfmod{\ell+2}{\slindrel{u-r,1}^+}},\\
&\dses {\sfmod{\ell}{\slindrel{r,s}^-}} {\sfmod{\ell}{\slproj{r, s}^-}}    {\sfmod{\ell-1}{\slindrel{r,s+1}^-}} ,\\
&\dses {\sfmod{\ell}{\slindrel{r,v-1}^-}} {\sfmod{\ell}{\slproj{r, v-1}^-}}    {\sfmod{\ell-2}{\slindrel{u-r,1}^-}} ,\\
\end{align*}
for $\ell \in\ZZ, r = 1, \dots,  u-1, s = 1, \dots, v-2$. 	
\end{prop}
All extensions except for the $\sfmod{\ell}{\slproj{r, v-1}^\pm}$ were constructed by Adamovi\'c in \cite{AdaRea17} . We give a uniform and new construction of all extensions in the Appendix. 

\section{Category $\sWtsl{k}_{\geq0}$ of lower bounded modules}
\label{sec:catlb}

In this Section, we describe the category $\sWtpg k$ of lower bounded modules over $\savoa k \g$, where $\g=sl_2$ and $k$ is an admissible level.

\subsection{The irreducible modules}\label{sec:loc}

The classification of irreducible modules in $\sWtsl{k}_{\geq0}$ is
obtained in \cite{AdaVer95} and generalized to arbitrary higher-rank finite-dimensional simple Lie algebra $\g$ in \cite{KawRid21}. We have presented the classification in \Cref{rhwsimples}
in the last Section.
In this Subsection, we describe the idea of the proof of (the higher rank generalization of)  \Cref{rhwsimples}.

Let $\g$ be a finite-dimensional simple Lie algebra.
Let  $V$ be either
$\uavoa k \g$ or $\savoa k \g$ with $\cR=\bigWtpg k$ or $\cR=\sWtpg k$, respectively.
We denote the Zhu algebra of $V$ by $Z(V)$.
For a bounded $V$-module $M$, let $M_{top}$ be  the space of the lowest conformal weight (the top space), which
 is a $Z(V)$-module.

\begin{theorem}\cite{ZhuMod96}
The map $M\mapsto M_{top}$ gives a one-to-one correspondence between the irreducible modules in $\cR$ and the irreducible weight modules of $Z(V)$ with finite-dimensional weight spaces.
\end{theorem}

Thus, the irreducible modules in $\cR$ are irreducible relaxed-highest weight modules induced by
the irreducible weight modules over $Z(V)$.

We now concentrate on the case $V=\uavoa k\g$. 
In this case, the classification of irreducible modules in $\cR=\bigWtpg{k}$
follows from the classification of irreducible weight modules of $\g$, which is established in \cite{Fer,MatCla00}.

We first recall Fernando's result.
Let $\p$ be a parabolic subalgebra of $\g$ with the nilradical $\gu$ and Levi factor $\gl$.
The Levi factor $\gl$ is called of {\em AC type} if the derived
algebra $[\gl,\gl]$ is a direct sum of simple Lie algebras of type A or C.

\begin{definition}
An irreducible weight module $M$ over $\g$ is {\em dense} if there is $\lam\in \h^*$ such that the weight support of $M$ fulfills the whole coset $\lam+Q$, that is, $M=\bigoplus\limits_{\mu\in Q}M_{\lam+\mu}$ and $M_{\lam+\mu}\neq 0$ for any $\mu\in Q$.
\end{definition}

The dense modules over semisimple Lie algebras are defined in a similar way.
For a reductive Lie algebra $\gl$, an irreducible $\gl$-module $N$ is called dense if it is dense
as a $[\gl,\gl]$-module.
Note that highest weight modules are never dense.

\begin{theorem}\cite{Fer}
Let $M$ be an irreducible weight $\g$-module.
Then there is a parabolic subalgebra $\p$ of $\g$ with
the Levi factor $\gl$ of AC type and a dense irreducible $\gl$-module $N$
such that $M\cong L_\p(N)$.
That is, any irreducible weight module is parabolically induced
by a dense module over a Levi subalgebra $\gl$ of AC type.
\end{theorem}

In particular, dense modules are the same as the {\em cuspidal} modules, the modules which are not properly parabolically induced.

Moreover, in \cite{MatCla00}, the cuspidal modules are classified.
Let $\g$ be of type $A_\ell$ or $C_\ell$.
The cuspidal modules are obtained by applying the {\em twisted localization} functor
to {\em uniformly bounded} irreducible highest weight modules.
Here, a weight module $M$ is called {\em uniformly bounded}
if the weight multiplicities $\dim M_\lam$ are bounded above when $\lam$ runs over $\h^*$.

To introduce the twisted localization functor, we prepare several notations.
Let us fix Cartan and Borel subalgebras $\h\subset \gb\subset \g$.
Let $\alpha_1,\ldots,\alpha_\ell$ and $\omega_1,\ldots,\omega_\ell$ be the simple roots
and fundamental weights.
For each root $\alpha$, we take root vectors $e^\alpha$ of root $\alpha$ and $f^\alpha$ of root $-\alpha$.
Let $S$ be the multiplicative subset of $U(\g)$ generated by $f^{\beta_1},\ldots,f^{\beta_\ell}$
with linearly independent positive roots $\beta_1,\ldots,\beta_\ell$ such that $f^{\beta_1},\ldots,f^{\beta_\ell}$ 
commute with each other.
(It is known that there are such roots, see \cite{MatCla00}).
Then we have the {\em Ore localization} $S^{-1}U(\g)$ of $U(\g)$ (see \cite{MatCla00} for more details.)
Let $M$ be a uniformly bounded weight module on which the elements of $S$
act injectively.
Then the localization $S^{-1}M$ over $S^{-1}U(\g)$ is non-zero and contains $M$ as a submodule.
Note that if $M$ is a dense irreducible module, then the elements of $S$ act bijectively on $M$
and $S^{-1}M=M$.
For any $\mu\in\h^*$, the following functor $\bm{f}^\mu$, called the twisted localization functor, shifts the weight support of
$S^{-1}M$ by $-\mu$:
$$
\bm{f}^\mu(M)=\{\bm{f}^\mu(v)\,|\,v\in S^{-1}M\}
$$ 
is a $\g$-module with the underlying vector space isomorphic to $S^{-1}M$ and 
the action
$$
a.\bm{f}^\mu(v)=\bm{f}^\mu((\bm{f}^{-\mu}a \bm{f}^\mu.v)\quad (a\in\g, v\in S^{-1}M), 
$$
where the twisting $\bm{f}^{-\mu} a \bm{f}^\mu\in S^{-1}U(\g)$ is defined through the exponential formula:
$$
\bm{f}^{-\mu}a \bm{f}^\mu=\sum_{i_1,\ldots,i_\ell\leq 0}\binom{-n_1}{i_1}\cdots
\binom{-n_\ell}{i_\ell}\ad(f^{\beta_1})^{i_1}\cdots \ad(f^{\beta_\ell})^{i_\ell}(a)(f^{\beta_1})^{-i_1}
\cdots (f^{\beta_\ell})^{-i_\ell},
$$
where $n_1,\ldots,n_\ell\in\C$ are defined by
$\mu=n_1\beta_1+\cdots n_\ell \beta_\ell$.
Here, $\ad(a)(b):=[a,b]$ for $a,b\in U(\g)$.
Note that if $v\in S^{-1}M$ is a weight vector of weight $\nu\in\h^*$, then
the element $\bm{f}^\mu(v)\in \bm{f}^\mu(M)$ is that of weight $\nu-\mu$.
We have $\bm{f}^\mu(M)\cong \bm{f}^\nu(M)$ if and only if  $\mu-\nu\in Q$.
So, we often write $\bm{f}^{\mu+Q}(M)=\bm{f}^\mu$
for $\mu+Q\in \h^*/Q$.

Now we explain the result of \cite{MatCla00}.
Let $\lam$ be an element of $\h^*$ of the form
$$
\lam=k_1\omega_1+\cdots k_\ell \omega_\ell,\quad k_1,\ldots,k_\ell\not\in \ZZ_{\geq0}.
$$
and suppose that $L(\lam)$ is uniformly bounded.
Then we have a certain subset $\Sing(\lam)\subset \h^*/Q$, which is a union of finite codimension one subsets of
$\h^*/Q$, and it follows that $\bm{f}^\mu(L(\lam))$ is cuspidal if and only if $\mu\not \in \Sing(\lam)$.
Moreover, all cuspidal modules are obtained in this way.

Thus, by combining \cite{Fer,MatCla00,ZhuMod96}, we have the classification of irreducible modules in $\bigWtpg{k}$.

We now briefly explain about the classification of irreducible modules in $\sWtpg{k}$ \cite{KawRid21},
where $k$ is an admissible level.
As we can expect from the above, it reduces to the classification of irreducible highest weight modules of $\savoa{k}{\g}$, which was solved in \cite{A12-2}.
See \cite{KawRid21} for more details.

\subsection{The indecomposable modules in $\sWtpg k$}
In this Subsection, we classify the indecomposable modules in $\sWtpg k$, where $k$ is an admissible level and $\g = \SLA{sl}{2}$.
The result is as follows.
Recall the modules  $\slindrel{r,s}^\pm$  in $\sWtpg k$ from
\Cref{subsecsl2basics}.

\begin{theorem}\label{thm:van_ext}
Let $\g = \SLA{sl}{2}$. The indecomposable but not irreducible modules in $\sWtpg k$
are isomorphic to $\slindrel{r,s}^\pm$ with $r=1,\ldots,u-1$ and $s=1,\ldots,v-1$.
\end{theorem}

To show the theorem, it suffices to show the following:
for any $1\leq r,r'\leq u-1, 1\leq s\leq v-1$, we have
\begin{align}
&\Ext^1(\mc L_r,\mc L_{r'})=\Ext^1(\mc L_r',\mc D^\pm_{r,s})=
\Ext^1(\mc D^\pm_{r,s},\mc L_r')=0,\label{eqn:rel1}\\
&\Ext^1(\mc D^+_{r,s},\mc D^+_{r',s'})=
\Ext^1(\mc D^-_{r,s},\mc D^-_{r',s'})=0,\label{eqn:rel2}\\
&\Ext^1(\mc D^+_{r,s},\mc D^-_{r',s'})=
\Ext^1(\mc D^-_{r,s},\mc D^+_{r',s'})=
\begin{cases}
\C&(r'=u-r,s'=v-s),\\
0&(\mbox{otherwise}),
\end{cases}
\label{eqn:rel3}\\
&\Ext^1(\mc D^\pm_{r,s},\mc E^\pm_{r',s'})=
\Ext^1(\mc E^\pm_{r,s},\mc D^\pm_{r',s'})=
\Ext^1(\mc E^\pm_{r,s},\mc E^\pm_{r',s'})=0,\label{eqn:rel4}
\end{align}
where double signs are not necessarily in the same order, 
and moreover for $\lam\not\in (\lam_{r,s}+Q)\cup(\lam_{u-r,v-s}+Q)$ and $\mc M_{r',s'}\in \{\mc D^\pm_{r',s'}, \mc E^\pm_{r',s'}\}$,
\begin{align}
&\Ext^1(\mc E_{\lam,\Delta_{r,s}},\mc L_{r'})=\Ext^1(\mc L_{r'},\mc E_{\lam,\Delta_{r,s}})=0,\quad 
\Ext^1(\mc E_{\lam,\Delta_{r,s}},\mc M_{r',s'})=\Ext^1(\mc M_{r',s'},\mc E_{\lam,\Delta_{r,s}})=0. \label{eqn:rel5}
\end{align}
By using the rationality of the category $\mc O_k$ of $\savoa k\g$,
we already have \eqref{eqn:rel1} and \eqref{eqn:rel2}.
To show \eqref{eqn:rel3}--\eqref{eqn:rel5},
we first show the following.

\begin{proposition}\label{prop:relext}
Let $\g$ be a finite-dimensional simple Lie algebra and $k$ be admissible.
Let $M,N$ be finite-length, almost simple, relaxed-highest weight $\savoa k\g$-modules with the lowest conformal weights $h_M$ and $h_N$.
Then the following hold. 
\begin{enumerate} \item If $h_M= h_N$, then $\Ext^1(M,N)\cong \Ext^1(M_{top},N_{top})$.
\item  If $h_M\neq h_N$, then $\Ext^1(M,N)=0$.
\end{enumerate}
\end{proposition}

For formal series $f(z,q)$ and $g(z,q)$, we write $f(z,q)\geq g(z,q)$
if it holds coefficient-wise.
\vspace{5mm}

\noindent { \bf Case 1:} The case of $h_M=h_N$.

Consider the natural linear map $\Ext^1(M,N)\rightarrow \Ext^1(M_{top},N_{top})$.
By the Zhu induction, it is a surjection.
Therefore, it suffices to show that the kernel of this map is zero, which follows from the following. 

\begin{proposition}\label{prop:same}
Let $A,C$ be almost simple, relaxed-highest weight modules over $\ag$.
Suppose that $A$ and $C$ have the same lowest conformal weight.
Let $B$ be a $\ag$-module with the exact sequence
\begin{equation}\label{eqn:seq1}
0\ra A\ra B\ra C\ra0.
\end{equation}
Suppose that the exact sequence
\begin{equation}\label{eqn:seq2}
0\ra A_{top}\ra B_{top}\ra C_{top}\ra0
\end{equation}
splits, that is, $B_{top}=A_{top}\oplus C_{top}$.
Then \eqref{eqn:seq1} splits.
\end{proposition}

\begin{proof}
Let $C'$ be the submodule of $B$ generated by $C_{top}$.
We first show that $A\cap C'=0$.
Let $v$ be a non-zero element of $A\cap C'$, on the contrary.
Since $v\in C'$, it has the form
$$
v=\sum_{i=1}^r a_iu_i,\qquad r\geq 1,\quad a_i\in U(\ag), \quad  u_i\in C_{top}.
$$
Since $A$ is almost simple, we have a non-zero element
$w\in A_{top}$ and $b\in U(\ag)$ such that
$
w=bv.
$
Then $w=\sum_i (ba_i)u_i$ and by the PBW theorem,
we have $c_i\in U(\g)$ such that $w=\sum_i c_iu_i$.
But it contradicts to \eqref{eqn:seq2}.
Therefore, we have $A\cap C'=0$.

We have just obtained $A\oplus C'\subset B$.

Let $M(C_{top})$ be the relaxed Verma module generated by $C_{top}$:
$$
M(C_{top})=U(\ag)\otimes_{U(\ag[t]\oplus \C K)}C_{top}.
$$
Then we have the surjective $\ag$-homomorphism
$M(C_{top})\ra C'$.
Since $C$ is the almost simple quotient of $M(C_{top})$,
we have also the surjective homomorphism $C'\ra C$.
Therefore, $\ch{C'}\geq \ch{C}$.

We have the inequalities
$$
\ch{B}\geq \ch{A}+\ch{C'}\geq \ch{A}+\ch{C}=\ch{B}.
$$
Thus, $\ch{C'}=\ch{C}$, which implies the splitting $C\cong C'$.
\end{proof}
\vspace{5mm}

\noindent {\bf Case 2:} $h_M\neq h_N$.

\begin{proposition}\label{sec:propdiff}
Let $A,C$ be finite-length relaxed-highest weight modules
over $\savoa k\g$ with different lowest conformal weights $h_A$ and $h_C$.
Let $B$ be a $\savoa k\g$-module with the exact sequence
\begin{equation}\label{eqn:different}
0\ra A\ra B\ra C\ra0.
\end{equation}
Then it splits.
\end{proposition}

\begin{lemma} \label{lem:big}
It holds if  $h_A>h_C$.
\end{lemma}

We first show the proposition using the lemma.

\begin{proof}[Proof of \Cref{sec:propdiff}]
 If $h_A<h_C$, then we take restricted duals so that
$$
0\ra C^\vee\ra B^\vee \ra A^\vee\ra 0.
$$
By \Cref{lem:big}, it splits. So, \eqref{eqn:different} also splits.
\end{proof}

From now on, we suppose $h_A>h_C$ and prove \Cref{lem:big}.
Suppose on the contrary that \eqref{eqn:different} does not split.

\begin{lemma}
If \eqref{eqn:different} does not split, then there is a non-splitting exact sequence
$0\ra A''\ra B''\ra C\ra0$ such that $A''$ is a simple module
and $B''$ is a relaxed-highest weight module.
\end{lemma}

\begin{proof}
Let $B'$ be the submodule generated by the top space of $B$.
Since $h_A>h_C$, we have $B_{top}\cong C_{top}$ and 
 the exact sequence
\begin{equation}\label{eqn:spltemp}
0\ra A\cap B'\ra B'\ra C\ra0
\end{equation}
is induced,
since $C$ is generated by its top space.

We show that it does not split. Suppose on the contrary
that \eqref{eqn:spltemp} splits: $\iota:C\hookrightarrow B'$.
Since $B'$ is a submodule of $B$, the embedding $\iota:C\hookrightarrow B$ is induced.
Since $\pi \circ \iota=\mathrm{id}_C$, it is a splitting of \eqref{eqn:different}.

Since \eqref{eqn:spltemp} is non-splitting, the module $A\cap B'$ is non-zero. Then take a maximal submodule $I$ of $A\cap B'$ and set $A''=A/I$ and $B''=B'/I$. We then have a 
non-split exact sequence
$$
0\ra A'' \ra B''\ra C\ra0
$$
with simple $A''$ and relaxed-highest weight $B''$.
\end{proof}

So, we may assume that $A$ is simple and $B$ is a relaxed-highest weight module, towards a contradiction.

\begin{lemma}
Without loss of generality, we may assume that the top space of $C$ is a simple $\g$-module.
\end{lemma}

\begin{proof}
Recall that $\pi:B\ra C$ induces the isomorphism
$\pi:B_{top} \cong C_{top}$ since $h_A>h_C$.
Let $M$ be a simple $\g$-submodule of $B_{top}$.
Set $B'=\langle M\rangle_{\ag}\subset B$ and $C'=\langle \pi(M)\rangle_{\ag}\subset C$.
Then exactly one of the following (1) or (2) holds:
\begin{align*}
&(1)\quad 0\ra A\ra B' \ra C'\ra0\mbox{ (non split) }\\
&(2)\quad 0\ra A\ra B/B' \ra C/C'\ra 0 \mbox{ (non-split) }.
\end{align*}
If (1) holds, we have obtained the lemma.
If (2) holds, then since the length of the top space of $C/C'$
is less than that of $C$, we can proceed by induction until
(1) holds or the top space of $C/C'$ becomes simple.
\end{proof}

So, we assume that the top space of $C$ is simple.

\begin{proof}[Proof of \Cref{sec:propdiff}]
We first show the case of $\g=\SLA{sl}{2}$ and then state how to
generalize the proof to the general cases.
By the above lemmas, we may  assume that $B$ is a relaxed-highest weight module
generated by the irreducible module $C_{top}$.
To deduce the contradiction, it suffices to show that the ``Verma" $\savoa k\g$-module $V(C_{top})$ generated by $C_{top}$ is irreducible: in this case, we have $V(C_{top})\cong B\cong C$ and $A=0$, which contradicts to the simpleness of $A$.

Let $\g=\SLA{sl}{2}$.
By the classification of irreducible weight modules of $\SLA{sl}2$, we have the following 2 possibilities:

\begin{enumerate}
\item $C_{top}$ is a highest weight module (replacing the Borel of $\g$ if necessary).
\item $C_{top}$ is a cuspidal simple module.
\end{enumerate}

First case: We see that $V(C_{top})$ is a highest weight module. 
Since $k$ is admissible, we see that
$V(C_{top})$ is irreducible.

Second case: 
In this case, we see that $f=ft^0,e=et^0$ acts bijectively on $V=V(C_{top})$.
Then we have $\lam\in \C$ such that the twisted localization $f^\lam(V)$ satisfies
$f^\lam(V)_{top}\cong E^+_{r,s}$ for some $r,s$.
Here, $E^+_{r,s}$ denotes $E^+_{\lam_{r,s},\Delta_{\lam_{r,s}}}$.
Suppose that $V$ is not simple.
Then we have a non-trivial relaxed-highest weight submodule $N\subset V$, whose lowest conformal weight is necessarily greater than that of $V$.
Since $f,e$ acts bijectively on $N$, the  top space $N_{top}$ is a direct sum of cuspidal simple modules.
By \cite[Lemma 13.12]{MatCla00}, we see that $f^\lam(V)$ is generated by $f^\lam(V)_{top}=f^\lam(V_{top})\cong E^+_{r,s}$.
Since $f^\lam$ is an exact functor, we have the non-trivial
submodule $f^\lam N$.
Since $f^\lam(V)_{top}$ is a sum of highest and lowest weight modules, cuspidal modules do not appear as a component of $f^\lam N_{top}$.
As $f^\lam N$ is of finite length and there is a $\mu\in \h^*$ such that the whole set $\mu+Q$ is included in the weight support of $f^\lam N_{top}$ because of the cuspidality of $N_{top}$, we see that there is a highest weight submodule of $f^\lam N_{top}$.
Therefore, there is a highest weight submodule $H$ in $f^\lam N$.

Then, the module $D$ generated by $D^+_{r,s}\subset f^\lam(V)_{top}$ is a highest weight module containing $H$.
Here, $D^+_{r,s}$ denotes $D^+_{\lam_{r,s}}$.
Since $k$ is admissible, the module $D$ is simple, which contradicts that $D$ contains $H$.
Therefore, $V=V(C_{top})$ is irreducible.

Thus, \eqref{eqn:different} splits.

We now let $\g$ be an arbitrary finite-dimensional simple Lie algebra.
Let us recall the classification of weight modules from \cite{MatCla00}.
By replacing the Borel of $\g$ if necessary,
we have the following possibilities (1)' and (2)' instead of (1) and (2) in the $\g=\SLA{sl}{2}$ case:

\begin{enumerate}
\item[(1)'] $C_{top}$ is a highest weight module.
\item[(2)'] We have a parabolic subalgebra 
$\gb \subsetneq \p\subseteq \g$
 with the Levi subalgebra $\gl\subset \p$ of $\p$
 such that $C_{top}$ is a simple weight $\g$-module parabolically induced from a cuspidal module $M$ over $\gl$, namely,
 $$
 C_{top}\cong (U(\g)\otimes_{U(\p)}M)/J,
 $$
 where the nilradical of $\p$ acts trivially on $M$ and $J$ is the maximal proper submodule.
\end{enumerate}
In both cases, we have contradictions, and therefore \eqref{eqn:different} splits,
in a similar way to the case of $\g=\SLA{sl}{2}$.
We just give a brief remark on (2)', the twisted localization
functor for $\gl$ is also constructed in \cite{MatCla00} with mutually commuting nilpotent elements
$f^{(1)},\ldots, f^{(\ell)}\in \gl$, where $\ell$ is the rank of the
semisimple Lie algebra $[\gl,\gl]$.
\end{proof}

We return to the case of $\g=sl_2$ and proceed to
prove \eqref{eqn:rel3}--\eqref{eqn:rel5}.

\begin{lemma}\label{lem:delta}
Let $r,r',s,s'$ be integers such that $1\leq r,r'\leq u-1$, $0\leq s,s'\leq v-1$.
Then,
$$
\Delta_{r,s}=\Delta_{r',s'}\Longleftrightarrow
(r',s')=(r,s)\quad \mbox{or}\quad (r',s')=(u-r,v-s).
$$ 
\end{lemma}

(If $s=0$ or $s'=0$, then the condition after ``or" is never satisfied.)

\begin{proof}
\begin{align*}
&\Delta_{r,s}=\Delta_{r',s'}
\Longleftrightarrow
(r-ts)^2=(r'-ts)^2
\Longleftrightarrow
r-\frac uv s=r'-\frac uv s' \quad\mbox{or}\quad
r-\frac uv s=-r'+\frac uv s',
\end{align*}
and the last condition is equivalent to either
\begin{align}
&r-r'-\frac uv(s-s')=0,\label{eqn:deltaequal1}\\
&r+r'-\frac uv(s+s')=0.\label{eqn:deltaequal2}
\end{align}
Since $-v+1\leq s-s'\leq v-1$, the condition
\eqref{eqn:deltaequal1} forces $s=s'$.
Moreover, by $0\leq s+s'\leq 2v-2$, we see that
\eqref{eqn:deltaequal2} implies $s+s'=0,v$.
The rest is easy.
\end{proof}

\begin{proof}[Proof of \eqref{eqn:rel3}--\eqref{eqn:rel5}]
By \Cref{prop:relext} (2) and the above Lemma, it suffices to consider
$(r',s')=(r,s)$ or $(u-r,v-s)$.
By \Cref{prop:relext} (1), we only need to consider the extensions in top spaces.
Since the Casimir operator acts as scalar multiples on $Z(\savoa k\g)$-modules by \Cref{rem:centralchar}, we see that 
all cases other than $\Ext^1(D^+_{r,s},D^-_{u-r,v-s})$ and $\Ext^1(D^-_{r,s},D^+_{u-r,v-s})$ are zero.
Moreover, they are known to be one-dimensional.
Thus, we have proved \eqref{eqn:rel3}--\eqref{eqn:rel5}.
\end{proof}

\subsection{Blocks}
Let $\g=sl_2$ and $k$ an admissible level.
By the above results, we have the block decomposition of $\sWtpg k$.
\begin{enumerate}
\item
Atypical blocks: $[\mc L_r]$ ($1\leq r\leq u-1$), $[\mc D^+_{r,s}]$
($1\leq r\leq u-1,1\leq s\leq v-1$).
\item
Typical blocks: $[\mc E_{\lam,\Delta_{r,s}}]$ ($(r,s)\in I_{u,v}$, $\lam\not\in (\lam_{r,s}+Q)\cup(\lam_{u-r,v-s}+Q)$).
\item
The module $\mc D^-_{r,s}$ ($1\leq r\leq u-1, 1\leq s\leq v-1$)
belongs to the block $[\mc D^+_{u-r,v-s}]$.
\end{enumerate}

\section{The category of weight modules}
\label{sec:catwt}

\subsection{Criteria on extensions}

Let $\mathfrak{g}$ be a simple Lie superalgebra and $\widehat{\mathfrak{g}}$ be its affinization. 
Let $M$ be an indecomposable module 
in $\bigWt{k}$
so that there exists $h_M \in \mathbb C$ and $\mu_M \in \mathfrak{h}^*$, such that
\begin{equation}
M = \bigoplus_{\substack{ \lambda \in \mu_M + Q \\ \Delta\in h_M + \mathbb Z}}  M_{\lambda, \Delta}
\end{equation}
Let $\Pi = \{ \alpha_1, \dots, \alpha_n\}$ be a set of simple positive roots for $\mathfrak{g}$ and $\Delta_+$ the corresponding set of positive roots. Denote by $\theta$ the longest positive root. 
We say that $M$ is {\em{of type $\Pi$}} if for every $\Delta$ in $h_M + \mathbb Z$ there exists a $\mu(\Delta)$, such that $M_{\mu(\Delta), \Delta} \neq 0$, but 
$M_{\mu(\Delta)-\beta, \Delta} = 0$ for $\beta$ in the cone $C(\Pi) = \mathbb Z_{\geq 0} \alpha_1 \oplus \dots \oplus  \mathbb Z_{\geq 0} \alpha_n \setminus \{ 0\}\subset Q$. 
(See \Cref{fig:pi-type}.)
By definition, a type $\Pi$ object in $\bigWt{k}$ never belongs to $\bigWtpg{k}$.

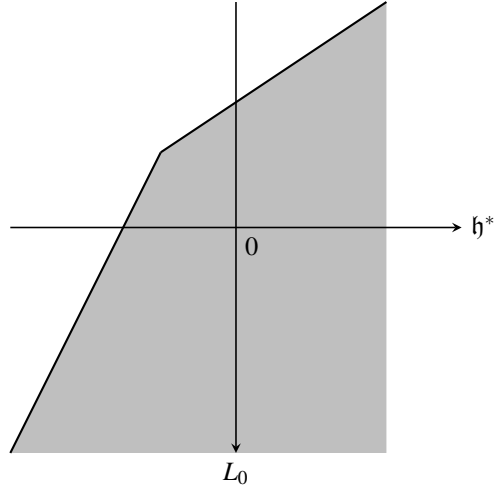
\begin{figure}
\begin{tikzpicture}
\coordinate (A) at (2,3);
\coordinate (B) at (-1,1);
\coordinate (C) at (-3,-3);
 \fill[lightgray] (A)--(B)--(C)--(2,-3)--(A);
 \draw[semithick,->,>=stealth] (-3,0)--(3,0) node[right] {$\h^*$}; 
 \draw[semithick,->,>=stealth] (0,3)--(0,-3) node[below] {$L_0$}; 
 \draw (0,0) node[below right]{$0$}; 
 \draw[thick](A)--(B);
 \draw[thick](B)--(C);
\end{tikzpicture}
 \caption{A $\Pi$-type module, where $\Pi=\{\alpha\}$. The weight support is indicated in gray color.}
 \label{fig:pi-type}
 \end{figure}

A trivial criterion for vanishing extensions is
\begin{prop}
For two indecomposable objects $M, N$ and  in $\mathcal C$ and $n \in \mathbb Z_{\geq 0}$ one has $\Ext^n(M, N) \neq 0$ implies $\mu_N = \mu_M \mod Q$ and $h_M = h_N \mod \mathbb Z$. 
\end{prop}
Here $\Ext^0 = \Hom$. We derive a new criterion now.

Let $k$ be an admissible level and $\g=\SLA{sl}{2}$.

\begin{lemma}
Let $N$ be a simple module in $\sWtsl{k}$ 
of $\Pi$-type.
Let $h$ be a complex number and $\lam$ an element of $\h^*$.
Then there is $\mu\in\lam+Q$ such that
$N_{\mu-n\alpha,h-m}=0$ for any $n,m\geq 0$.
\end{lemma}

\begin{proof}
The module $N$ is isomorphic to one of the following:
\begin{align*}
&\sigma^\ell(\mc{D}^{{+}}_{r,s}), \quad 1\leq r\leq u-1,1\leq s\leq v-2,\ell\leq -2,\\
&\sigma^\ell(\mc{D}^{{+}}_{r,v-1}),\quad 1\leq r\leq u-1,\ell\leq -3,\\
&\sigma^\ell(\cE_{\lam;\Delta_{r,s}}),\quad 
1\leq r\leq u-1,1\leq s\leq v-1,\ell\leq -1, {
\lam\neq \lam_{r,s},\lam_{u-r,v-s}
}.
\end{align*}
They all have the desired property.
\end{proof}

\begin{theorem}\label{thm:extaaa}
Let $M$ be either
$$
\mc{D}^+_{r,s},\quad \cE_{\lam;\Delta_{r,s}},\quad \cE_{r,s}^-
$$ 
{\rm(}$1\leq r\leq u-1$, $1\leq s\leq v-1$, $\lam\neq \lam_{r,s},\lam_{u-r,v-s}${\rm)}
and $N$ simple and of $\Pi$-type.
Then $\mathrm{Ext}^1(M,N)=0$.
\end{theorem}

\begin{proof}
Let $\Delta$ be the lowest conformal weight of $M$.
Suppose on the contrary that there is a non-split exact sequence
$$
0\ra N\ra\tilde M\ra M\ra0
$$
with a module $\tilde M$.
By the conditions, $\tilde M$ is generated by a single vector $v\in \tilde M_{\mu-\alpha,\Delta}$
for some $\mu\in\h^*$ such that $N_{\mu-n\alpha,\Delta-m}=0$
for all $n,m\geq 0$.
Since $N$ is of $\Pi$-type, we have a non-zero weight vector $w$ of $N_{\Delta-1}$.
Then we have  $x\in U(\ag)$ such that $w=xv$.
(See \Cref{fig:vw}.)
By the PBW theorem, we see that $w$ is the sum of the PBW monomials of the form
$$
x^1_{(-n_1)}\ldots x^r_{(-n_r)}.y^1_{(0)}\ldots y^s_{(0)}.
z^1_{(m_1)}\ldots z^u_{(m_u)}v
$$
with $r,s,u\geq 0$, $x^i,y^j,z^l\in\{e,h,f\}$, $n_i,m_j>0$ such that
$-\sum_{i=1}^r n_i+\sum_{j=1}^u m_j=1$.
Here, we set $x_{(n)}=xt^n$ for any $x\in\g$ and $n\in\ZZ$.
Since $-\sum_{i=1}^r n_i+\sum_{j=1}^u m_j=1$, we see that $u>0$.
However, as we shall see below, $z^u_{(m_u)}v=0$, which contradicts to $w\neq0$.

First, $z^u_{(m_u)}v\in N$ since it is of conformal weight less
than $\Delta$, say, $\Delta-m_u$.
Second, since $N_{\mu-n\alpha,\Delta-m_u}=0$ for $n\geq0$ and the $h_0$-weight of $z^u_{(m_u)}v$ is either 
$\mu-\alpha,\mu-2\alpha$ or $\mu$,
we have $z^u_{(m_u)}v=0$.

Thus, we see that there is no such $\tilde M$.
\end{proof}

\begin{figure}
\begin{tikzpicture}
\coordinate (A) at (2+1/2,3);
\coordinate (B) at (1/2,-3);
\coordinate(C) at (-1+1/3,1/2);
\coordinate(D) at (2,-3);
\coordinate(v) at (-2,1/2);
\coordinate(mu-alpha) at (-2,0);
\coordinate(mu) at (-2+1/2+1/4,0);
\coordinate(delta) at (0,1/2);
\coordinate(w) at (2+1/3,1+1/2);
\coordinate(delta-1) at (0,1+1/2);
\coordinate(x) at (1,1+1/4);
 \fill[lightgray] (A)--(B)--(2+1/2,-3)--(A);
 \fill[lightgray] (C)--(D)--(-3,-3)--(-3,1/2)--(C); 
 \draw[semithick,->,>=stealth] (-3,0)--(3,0) node[right] {$\h^*$}; 
 \draw[semithick,->,>=stealth] (0,3)--(0,-3) node[below] {$L_0$}; 
 \draw[thick](A)--(B); 
 \fill (delta) circle (1pt) node[right]{$\Delta$}; 
 \fill (delta-1) circle (1pt) node[left]{$\Delta-1$};
\fill (mu-alpha) circle (1pt) node[below]{$\mu-\alpha$};
\fill (mu) circle (1pt) node[below]{$\mu$};
 \draw[thick](-3,1/2)--(C)--(D);
 \fill (v) circle (2pt) node[above]{$v$};
 \fill (w) circle (2pt) node[below]{$w$};
 \draw[dotted, semithick](v)--(delta);
 \draw[dotted, semithick](v)--(mu-alpha);
 \draw[dotted, semithick](w)--(delta-1);
 \draw[dashed, semithick,->>, >=stealth](v)--(w);
 \draw(x) node[below]{$x$};
 \draw(-2,-2) node[color=white]{$M$};
 \draw(2,-1) node[color=white]{$N$};
\end{tikzpicture}
 \caption{Picture for the proof of \Cref{thm:extaaa}, where $M=\sldis{r,s}^+$.}
 \label{fig:vw}
 \end{figure}

%

By considering the Weyl group action $\alpha\leftrightarrow-\alpha$, we have the following corollary.

\begin{cor}\label{cor:ext} Let $\Pi^\pm = \{ \pm \alpha\}$
\begin{enumerate}
\item $\Ext^1(\slrel{\lambda}{r,s}, M) = 0=\Ext^1(M,\slrel{\lambda}{r,s})$  for $\lambda \neq \lambda_{r, s}, \lambda_{u-r, v-s}$ and $M$ any simple module. 
\item $\Ext^1(\slindrel{r,s}^-, M) = 0$ for $M$ any simple module of $\Pi^+$-type in particular for $M = \sigma^{-\ell}(\sldis{r',s'}^+)$ 
with $r=1,\ldots,u-1$ and $s=1,\ldots, v-2$, $\ell \geq 2$
or $s=v-1$, $\ell\geq 3$.
\item $\Ext^1(\slindrel{r,s}^+, M) = 0$ for $M$ any simple module of $\Pi^-$-type in particular for $M = \sigma^{\ell}(\sldis{r',s'}^-)$ 
with $r=1,\ldots,u-1$ and $s=1,\ldots, v-2$, $\ell \geq 2$
or $s=v-1$, $\ell\geq 3$.
\item $\Ext^1(\sldis{r,s}^+, M) = 0$ for $M$ any simple module of $\Pi^+$-type in particular for $M = \sigma^{-\ell}(\sldis{r',s'}^+)$ 
with $r=1,\ldots,u-1$ and $s=1,\ldots, v-2$, $\ell \geq 2$
or $s=v-1$, $\ell\geq 3$.
\item $\Ext^1(\sldis{r,s}^-, M) = 0$ for $M$ any simple module of $\Pi^-$-type in particular for $M = \sigma^{\ell}(\sldis{r',s'}^-)$ 
with $r=1,\ldots,u-1$ and $s=1,\ldots, v-2$, $\ell \geq 2$
or $s=v-1$, $\ell\geq 3$.
\end{enumerate}
\end{cor}

\begin{proof}
We only show (1). By using the dual operator $\vee$, we only need to show the first equality. If $M$ is lower bounded, it follows by \eqref{eqn:rel5}.
If  $M$ is not lower bounded, $M$ is of $\Pi^+$- or $\Pi^-$-type.
If $M$ is of $\Pi^+$-type, the assertion is nothing but \Cref{thm:extaaa}.
Otherwise, by replacing the Borel subalgebra, we can apply \Cref{thm:extaaa} as well. The rest is shown in a similar way.
\end{proof}

We have the following corollary on the extensions between
irreducible modules.

\begin{corollary}\label{cor:simples}
Let $M,N$ be simple objects in 
$\sWtsl{k}$.

\begin{enumerate}
\item  If $\Ext^1(N,M)\neq0$, then the following condition $(*)$ holds:
 
 $(*)$: There are $\ell\in\ZZ$, 
$1\leq r\leq u-1$, $1\leq s\leq v-1$ and $\varepsilon\in\{+,-\}$ such that
$M\cong\sigma^\ell(\mc{D}^\varepsilon_{r,s})$, 
$N\cong\sigma^\ell(\mc{D}^{-\varepsilon}_{u-r,v-s})$.

\item If $(*)$ holds, then $\Ext^1(N,M)=\C$ and the non-zero element is isomorphic to 
 $\sigma^\ell(\cE^{\varepsilon}_{r,s})$.
\end{enumerate}
\end{corollary}

\begin{proof}
The second assertion (ii) is known.
Now, suppose that we have a non-split extension $\Ext^1(N,M)\neq 0$.
By \Cref{cor:ext} (1), we see that neither $M$ nor $N$ is not of $\cE$-type.
Then we have an integer $n\in\ZZ$ such that  $\sigma^n(N)\cong \cD^+_{r,s}$ with
$1\leq r\leq u-1$ and $1\leq s\leq v-1$.
Since $\Ext^1(N,M)\cong \Ext^1(\sigma^n(N),\sigma^n(M))$,
we have $\Ext^1(\mc{D}^+_{r,s},\sigma^n(M))\neq 0$.
Since $\sigma^n(M)$ is lower bounded or of $\Pi^\pm$-type,
we will split the proof in these 3 cases.

If $\sigma^n(M)$ is lower bounded, by \eqref{eqn:rel2} and \eqref{eqn:rel3},
we have $\sigma^n(M)\cong \sigma^n(\mc{D}^-_{u-r,v-s})$.
Therefore, we have $(*)$.

If $\sigma^n(M)$ is of $\Pi^+$-type, then \Cref{thm:extaaa} implies
$\Ext^1(N,M)=0$, which is a contradiction.

Suppose that $\sigma^n(M)$ is of $\Pi^-$-type.
We have $m<0$ such that $\sigma^m(\sigma^n(M))\cong \cD^+_{r',s'}$ with $1\leq r'\leq u-1$ and $1\leq s'\leq v-1$.
Then $\sigma^m(\mc{D}^+_{r,s})$ is either lower bounded or of type $\Pi^+$.
If $\sigma^m(\mc D^+_{r,s})$ is of lower bounded, then by \eqref{eqn:rel2} and \eqref{eqn:rel3}, we have $(*)$.
Suppose that $\sigma^m(\mc{D}^+_{r,s})$ is of $\Pi^+$-type. Since 
$\Ext^1(N,M)\cong \Ext^1(\sigma^{m+n}(N),\sigma^{m+n}(M))$, we have a non-split extension $\tilde M$ of the form
$$
0\ra\mc{D}^+_{r',s'}\ra \tilde M\ra \sigma^m(\mc{D}^+_{r,s})\ra0.
$$
By taking the dual  of the sequence, we have the non-split sequence
$$
0\ra\sigma^m(\mc{D}^+_{r,s})\ra \tilde M^\vee\ra \mc{D}^+_{r',s'}\ra0,
$$
which contradicts to \Cref{thm:extaaa}.
\end{proof}

\subsection{Blocks of $\sWtsl{k}$}

Let $\g=\SLA{sl}{2}$ and $k$ an admissible level.
By using \Cref{cor:simples}, we are ready to describe blocks of $\sWtsl{k}$. 

First, any spectral flow of each $\cE$-type irreducible module generates a block, called 
a typical block.
To be precise, let  $E^n_{r,s,\lam}$ with $n\in\ZZ$, $(r,s)\in I_{u,v}$ and 
$\lam\in \C/2\ZZ$ with $\lam\neq \lam_{r,s},\lam_{u-r,v-s}$ mod 2,
be the full subcategory of  $\sWtsl{k}$ 
 of all objects whose composition factors are isomorphic to
$\sigma^n\cE_{\lam;\Delta_{r,s}}$.
Then $E^n_{r,s,\lam}$ is a block of  $\sWtsl{k}$. 

We now introduce the rest of blocks, called the atypical blocks.
Since all spectral flows of $\mc E$-type modules are contained in typical blocks,
the irreducible modules in atypical blocks are atypical irreducible
modules, that is, the spectral flows of  $\mc D$-type
and $\mc L$-type modules.
It follows from \Cref{twistRules} that the inequivalent atypical irreducible modules are precisely
$\sigma^n(\mc D^+_{r,s})$ with $n\in\ZZ$, $1\leq r\leq u-1$
and $1\leq s\leq v-1$.

\begin{theorem}\label{atypicalblocks}
Each atypical block contains a unique highest weight module
and the set of all atypical blocks are parametrized by the highest weight module it contains.
In particular, there are precisely $(u-1)v$ different atypical blocks.
\end{theorem}

To show the theorem, let us consider the block $C_r$ containing $\mc D^+_{r,v-1}$ with $1\leq r\leq u-1$.
It turns out that the spectral flows $\sigma^n (C_r)$ with $1\leq r\leq u-1$ and $-1\leq n\leq v-2$ exhaust all atypical blocks and that
$\sigma^{-1}(C_r)$ contains $\mc L_{u-r}$ and $\sigma^n (C_r)$ with $0\leq n\leq v-2$ contains $\mc D^+_{r,v-1-n}$.
Therefore, all irreducible highest weight modules are contained in 
the blocks $\sigma^n (C_r)$.

Let $r=1,\ldots,u-1$.
The block $C_r$ is determined as follows.
By \eqref{twistRules} and \eqref{eqn:rel1}-\eqref{eqn:rel3}, the irreducible modules having non-zero $\mathrm{Ext}^1$ with $L_0:=\mc D^+_{r,v-1}$ are precisely
the following $L_{\pm 1}$:
\begin{eqnarray*}
L_1:=\begin{cases}
\sigma^{-1}(\mc D^+_{r,v-2}) & (v\geq 3),\\
\sigma^{-2} (\mc D^+_{u-r,1}) & (v=2),
\end{cases}
\qquad 
L_{-1}:=
\sigma^2(\mc D^+_{u-r,1}).
\end{eqnarray*}
By using \eqref{twistRules} and \eqref{eqn:rel1}-\eqref{eqn:rel3} again, $L_1$ has the unique irreducible atypical module $L_2$ not isomorphic to $L_0$ but having non-zero $\mathrm{Ext}^1$ with
$L_1$.
Similarly, $L_{-1}$ has the unique irreducible atypical module $L_{-2}$ not isomorphic to $L_0$ but having non-zero $\mathrm{Ext}^1$ with
$L_{-1}$.
By continuing this procedure, we define a set of irreducible 
atypical modules $\{L_n\}_{n\in\ZZ}$ and they form all irreducible modules in $C_r$.
We here write $L_0,L_1,\ldots,L_{v-1}$ in this order:
$$
\mc D^+_{r,v-1}, \sigma^{-1} (\mc D^+_{r,v-2}),
\sigma^{-2} (\mc D^+_{r,v-3}),\ldots,
\sigma^{-v+2} (\mc D^+_{r,1}),
\sigma^{-v} (\mc D^+_{u-r,v-1})
$$
It indicates the cyclic property modulo $v$ if we ignore the swap $r\leftrightarrow u-r$.
This swap is resolved when we consider all $1\leq r\leq u-1$.

Now, we see that $\sigma^1 (C_r)$ contains the irreducible modules
$$
\sigma^1(\mc D^+_{r,v-1}),  \mc D^+_{r,v-2},
\sigma^{-1} (\mc D^+_{r,v-3}),\ldots,
\sigma^{-v+3} (\mc D^+_{r,1}),
\sigma^{-v+1} (\mc D^+_{u-r,v-1})
$$
and notice that $\sigma^1 (C_r)$ contains the irreducible highest weight module $\mc D^+_{r,v-2}$.
In a similar way, it follows that $\sigma^n (C_r)$ with $0\leq n\leq v-2$ contains $\mc D^+_{r,v-1-n}$.
Moreover, since $\sigma^{-1}(\mc D^+_{r,v-1})\cong \mc L_{u-r}$, it 
follows that $\sigma^{-1} (C_r)$ contains
$$
\mc L_{u-r}, \sigma^{-2} (\mc D^+_{r,v-2}),
\sigma^{-3} (\mc D^+_{r,v-3}),\ldots,
\sigma^{-v+1} (\mc D^+_{r,1}),
\sigma^{-v} (\mc L_{r}).
$$
We now conclude that each irreducible highest weight module
is included in exactly one of $\sigma^n (C_r)$.
By the $v$-cyclic property, we clearly see that all atypical
irreducible modules $\sigma^m (\mc D^+_{r,s})$ with 
$m\in\ZZ$, $1\leq r\leq u-1$ and $1\leq s\leq v-1$ are included in
some $C^n_r$.
Thus, we have shown \Cref{atypicalblocks}.

We now describe more on an atypical block $C=\sigma^n (C_r)$.
Let $m$ be an integer and we rewrite $\sigma^n (L_m)$ as $L_m$.
 Note that the only non-zero $\Ext^1$-groups among irreducible modules in $C$ are of the form
\begin{equation}\label{extL}
\Ext^1(L_n,L_{n\pm1})=\C\quad n\in\ZZ.
\end{equation}
 We choose non-zero elements $E_m^\pm$ of $\Ext^1(L_m,L_{m\pm1})$:
$$
0\ra L_m\ra E_m^\pm \ra L_{m\pm1}\ra 0\quad \mbox{(non-split exact sequence)}.
$$
By writing $L_m\cong \sigma^{\ell_m}(\mc{D}^+_{r_m,s_m})$ with $1\leq r_m\leq u-1$,
$1\leq s_m\leq v-1$ and $\ell_m\in\ZZ$, we have $L_{m+1}\cong \sigma^{\ell_m}(\mc{D}^-_{u-r_m,v-s_m})$ and $E_m^+\cong \sigma^{\ell_m}(\cE^+_{r_m,s_m})$.
Similarly, we see that the module $E^-_m$ is a spectral flow of a $\cE^-$-type module.

\begin{lemma}\label{lem:Ext1}
For any $m\in\ZZ$,
\begin{equation}
\Ext^1(E^+_m, L_{m-1}) = 0 = \Ext^1(L_{m+2}, E^+_m), \qquad \Ext^1(E^-_m, L_{m+1}) = 0 = \Ext^1(L_{m-2}, E^-_m)
\end{equation}
\end{lemma}

\begin{proof}
We show the first and second equality.
We see that $\sigma^{-\ell_m}(L_m)\cong \mc{D}^+_{r_m,s_m}$ and
$\sigma^{-\ell_m}(L_{m+1})\cong \mc{D}^-_{u-r_m,v-s_m}$
are lower bounded modules.
Then, we see that $\sigma^{-\ell_m}(L_{m-1})$ is of $\Pi^-$-type
and $\sigma^{-\ell_m}(L_{m+1})$ is of $\Pi^+$-type.
It then follows from \Cref{cor:ext} (3) that 
$$
\Ext^1(\sigma^{-\ell_m}(E^+_m), \sigma^{-\ell_m}(L_{m-1})) = 0.
$$
It is equivalent to the first equality of the lemma.
Next, by \Cref{cor:ext} (2), we have
$$
\Ext^1(\cE^-_{r_m,s_m}, \sigma^{-\ell_m}(L_{m+2})) = 0.
$$
By taking the dual $\vee$, we have
$$
\Ext^1( \sigma^{-\ell_m}(L_{m+2}),\cE^+_{r_m,s_m}) = 0.
$$
Since the spectral flows preserve ext-groups, we have the second equality
of the lemma.
The rest are shown in a similar fashion.
\end{proof}

\begin{lemma}\label{lem:Ext2}
For any $m\in\ZZ$
\begin{equation}
\Ext^1(E^+_{m}, L_{m+1}) = 0 = \Ext^1(L_m, E^+_m), \qquad \Ext^1(E^-_{m}, L_{m-1}) = 0 = \Ext^1(L_m, E^-_m)
\end{equation}
\end{lemma}

\begin{proof}
We only show the second equality.
By applying $\sigma^{-\ell_m}$, we see that it is equivalent to
$$
\Ext^1(\mc{D}_{r,s}^+,\cE^+_{r,s})=0,
$$
where $r=r_m$ and $s=s_m$.
Assume on the contrary that there is a non-split exact sequence
$$
0\ra \cE^+_{r,s}\ra M\ra \mc{D}^+_{r,s}\ra0.
$$
By \Cref{prop:same}, we have non-split sequence
$$
0\ra \bar \cE^+_{r,s}\ra M_{top}\ra \bar{\mc{D}}^+_{r,s}\ra0,
$$
where $\bar{\mc{D}}^+_{r,s}$ is the irreducible highest-weight module of $\SLA{sl}{2}$ of highest weight $\lam_{r,s}$ and $\bar\cE^+_{r,s}$ is a dense $\SLA{sl}{2}$-module with non-split exact sequence
$$
0\ra \bar{\mc{D}}^+_{r,s}\ra \bar\cE^+_{r,s}\ra\bar{\mc{D}}^-_{r,s}\ra0.
$$
Here, $\bar{\mc{D}}^-_{r,s}$ is the $w$-conjugate of $\bar{\mc{D}}^+_{r,s}$,
which is the irreducible lowest weight module of lowest weight $-\lam_{r,s}$.
The $\SLA{sl}{2}$-module $M_{top}$ has three composition factors and both head
and socle of $M_{top}$ are isomorphic to $\bar{\mc{D}}^+_{r,s}$.
Now, we have a vector $v\in M_{top}$ such that the image of $v$
under the canonical surjection $M_{top}\ra \bar{\mc{D}}^+_{r,s}$ is a highest-weight vector, (which is non-zero).
We apply the Casimir operator $\Omega$ to $v$.
Then we see that $\Omega v=c v+d w$, with $c,d\in\C$, $d\neq 0$
and such that $w$ is a highest weight vector of the socle $\bar{\mc{D}}^+_{r,s}$.
It contradicts to the fact that the Casimir operator of $Z(L_k(\SLA{sl}{2}))$ acts as scalar on
indecomposable modules, Remark \ref{rem:centralchar}.
Thus, we have the second equality of the lemma.
The rest are proved in a similar way.
\end{proof}
The Adamovic construction of logarithmic modules, Proposition \ref{prop:log}, is in this notation
\begin{lemma}\label{lem:Ext3} \cite{AdaRea17} and Appendix \ref{appendix}.
For any $m\in\ZZ$,
\begin{equation}
\Ext^1(E^+_m, E^+_{m+1}) \neq 0 \qquad \text{and} \qquad \Ext^1(E^-_m, E^-_{m-1}) \neq 0.
\end{equation}\end{lemma}

\section{The atypical block}\label{sec:atyp}

Let $\mathcal C$ be a category whose simple inequivalent objects are $L_n$ for $n \in \mathbb Z$ 
{and whose objects have finite lengths}. We require that the properties of the atypical block of $\sWtsl{k}$
hold. These are
\begin{Properties}
$\mathcal C$ is a category whose simple inequivalent objects are $L_n$ for $n \in \mathbb Z$ and 
\begin{enumerate}
\item  By \eqref{extL}
\begin{equation}\label{extLL}
\Ext^1(L_n, L_m) = \begin{cases} \mathbb C & \quad  \text{if} \ |n-m| =1 \\ 0 &\quad \text{else} \end{cases}
\end{equation}
\item Let
\begin{equation}
 \ses{L_n}{E^\pm_n}{L_{n \pm1}} \qquad \in  \ \Ext^1( L_{n \pm 1}, L_n)
\end{equation}
{be a non-zero element.}
By Lemma \ref{lem:Ext1}
\begin{equation}\label{van1}
\Ext^1(E^+_{n+1}, L_n) = 0 = \Ext^1(L_n, E^+_{n-2}), \qquad \Ext^1(E^-_{n-1}, L_n) = 0 = \Ext^1(L_n, E^-_{n+2})
\end{equation}
\item 
By Lemma \ref{lem:Ext2}
\begin{equation}\label{van2}
\Ext^1(E^+_{n-1}, L_n) = 0 = \Ext^1(L_n, E^+_n), \qquad \Ext^1(E^-_{n+1}, L_n) = 0 = \Ext^1(L_n, E^-_n)
\end{equation}
\item By Lemma \ref{lem:Ext3}
\begin{equation}\label{ada}
\Ext^1(E^+_n, E^+_{n+1}) \neq 0 \qquad \text{and} \qquad \Ext^1(E^-_n, E^-_{n-1}) \neq 0.
\end{equation}
\end{enumerate}
\end{Properties}
In general, let $A, B, C, X$ be objects in $\mathcal C$ and 
\[
\ses{A}{B}{C}, \qquad \in \ \Ext^1(C, A)
\]
then there are the two long exact sequences 
\begin{equation}\label{long}
\begin{split}
 0 \rightarrow \Hom(X, A) \rightarrow \Hom(X, B) \rightarrow \Hom(X, C) \rightarrow \Ext^1(X, A) \rightarrow \Ext^1(X, B) \rightarrow \Ext^1(X, C) \rightarrow  \cdots \\
 0 \rightarrow \Hom(C, X) \rightarrow \Hom(B, X) \rightarrow \Hom(A, X) \rightarrow \Ext^1(C, X) \rightarrow \Ext^1(B, X) \rightarrow \Ext^1(A, X) \rightarrow  \cdots
\end{split}
\end{equation}
which we will now employ to determine extensions in $\mathcal C$.

\begin{theorem} \label{main} The atypical block satisfies
\begin{enumerate}
\item $\Ext^1(L_{n-1}, E^+_n) =\mathbb C$ and $\Ext^1(L_{m}, E^+_n) = 0$ for $m \neq n-1$.
\item $\Ext^1(E^+_n, L_{n+2}) =\mathbb C$ and $\Ext^1(E^+_n, L_m) = 0$ for $m \neq n+2$.
\item $\Ext^1(L_{n+1}, E^-_n) =\mathbb C$ and $\Ext^1(L_{m}, E^-_n) = 0$ for $m \neq n+1$.
\item $\Ext^1(E^-_n, L_{n-2}) = \mathbb C$ and $\Ext^1(E^-_n, L_{m}) = 0$ for $m \neq  n-2$.
\item $\Ext^1(E^-_n, E^+_m) = 0 = \Ext^1(E^+_n, E^-_m)$.
\item $\Ext^1(E^+_n, E^+_{n+1}) = \mathbb C = \Ext^1(E^+_n, E^+_{n+2})$ and $\Ext^1(E^+_n, E^+_{m}) =0$ for $m \neq  n+1, n+2$.
\item $\Ext^1(E^-_n, E^-_{n-1}) = \mathbb C = \Ext^1(E^-_n, E^-_{n-2})$ and $\Ext^1(E^-_n, E^-_{m}) =0$ for $m \neq  n-1, n-2$.
\item There exists an indecomposable module $P_n$ satisfying the non-split exact seqeunces
\[
\ses{E^+_n}{P_n}{E^+_{n-1}} \qquad \text{and} \qquad
\ses{E^-_n}{P_n}{E^-_{n+1}}.
\]
\item $\Ext^1(P_n, L_m) = 0 = \Ext^1(L_m, P_n)$. In particular $P_n$ is projective and injective. Since $P_n$ is indecomposable it is the projective cover and injective hull of $L_n$. Its Loewy diagram is
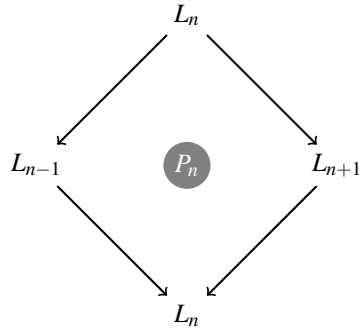
\begin{center}
\begin{tikzpicture}[scale=1]
\node (top) at (0,2) [] {$L_n$};
\node (left) at (-2,0) [] {$L_{n-1}$};
\node (right) at (2,0) [] {$L_{n+1}$};
\node (bottom) at (0,-2) [] {$L_n$};
\draw[->, thick] (top) -- (left);
\draw[->, thick] (top) -- (right);
\draw[->, thick] (left) -- (bottom);
\draw[->, thick] (right) -- (bottom);
\node (label) at (0,0) [circle, inner sep=2pt, color=white, fill=black!50!] {$P_n$};
\end{tikzpicture}
\end{center}
\item A complete list of inequivalent indecomposable modules is given in subsection \ref{sec:zigzag}
\item 
$\Ext^s(L_n, L_m) = \mathbb C$ if   $| n- m | \leq  s \ \text{and} \ n  -s  = m \mod  2$ and it is zero otherwise, i.e. $\Ext^\bullet(L_n, L_m) = x^{|n-m|}\mathbb C[x^2]$.
\end{enumerate} 
\end{theorem}
\begin{corollary}\label{maincor}
As an abelian category the atypical block is completely determined by Properties \ref{properties} and in particular equivalent to the atypical block of $u^H_q(\SLA{sl}{2})$.
\end{corollary}
We now prove the statements of the Theorem.

\subsection{The cases $\Ext^1(L_m, E^+_n)$ and $\Ext^1(E^+_n, L_m)$} 

${}$ \newline
The case $X = L_m$ and $\ses{L_n}{E^+_n}{L_{n +1}}$ in \eqref{long} yields
\begin{equation}
\begin{split}
\cdots \rightarrow  \Hom(L_m, E^+_n ) \rightarrow \Hom(L_m, L_{n+1} ) \rightarrow \Ext^1(L_m, L_n) \rightarrow \Ext^1(L_m, E^+_n) \rightarrow \Ext^1(L_m, L_{n+1}) \rightarrow \cdots \\
\cdots \rightarrow \Hom(E^+_n, L_m) \rightarrow \Hom(L_n, L_m) \rightarrow \Ext^1(L_{n+1}, L_m) \rightarrow \Ext^1(E^+_n, L_m) \rightarrow \Ext^1(L_n, L_m) \rightarrow  \cdots
\end{split}
\end{equation}
We see that if $m \notin \{n, n+2, n+1, n-1\}$ then $\Ext^1(L_m, E^+_n) = 0$. We have $\Ext^1(L_m, E^+_n) = 0$ for $m=n, n+2$ by \eqref{van1} and \eqref{van2}. If $m=n+1$, then the part of the long-exact sequence becomes
\[
\cdots \rightarrow  0 \rightarrow \mathbb C \rightarrow \mathbb C \rightarrow \Ext^1(L_{n+1}, E^+_n) \rightarrow  0  \rightarrow \cdots
\]
and hence 
$\Ext^1(L_{n+1}, E^+_n) = 0$. If $m=n-1$, then it collapses to 
\[
\cdots \rightarrow  0 \rightarrow 0 \rightarrow \mathbb C \rightarrow \Ext^1(L_{n-1}, E^+_n) \rightarrow  0  \rightarrow \cdots
\]
and hence $\Ext^1(L_{n-1}, E^+_n) =\mathbb C$. 
Let
\begin{equation}\label{Zn}
\ses{E^+_n}{Z_n}{L_{n-1}} \qquad \in \ \Ext^1(L_{n-1}, E^+_n)
\end{equation}
{be a non-zero element.}
Next we observe that $\Ext^1(E^+_n, L_m) = 0$ for $m \notin \{n-1, n+1, n, n+2 \}$. By \eqref{van1} and \eqref{van2} this extension also vanishes for $m= n\pm1$. For $m=n$ we also get zero, since
\[
\cdots \rightarrow  0 \rightarrow \mathbb C \rightarrow \mathbb C \rightarrow \Ext^1(E^+_n, L_n) \rightarrow  0  \rightarrow \cdots
\]
Finally the case $m=n+2$ gives $\Ext^1(E^+_n, L_{n+2}) = \mathbb C$, since
\[
\cdots \rightarrow  0 \rightarrow 0 \rightarrow \mathbb C \rightarrow \Ext^1(E^+_n, L_{n+2}) \rightarrow  0  \rightarrow \cdots
\]

\subsection{The cases ${\Ext^1(L_m, E^-_n)}$ and ${\Ext^1(E^-_n, L_m)}$}

${}$ \newline
The case $X = L_m$ and $\ses{L_n}{E^-_n}{L_{n -1}}$  in \eqref{long} yields
\begin{equation}
\begin{split}
\cdots  \rightarrow  \Hom(L_m, E^-_n ) \rightarrow \Hom(L_m, L_{n-1} ) \rightarrow \Ext^1(L_m, L_n) \rightarrow \Ext^1(L_m, E^-_n) \rightarrow \Ext^1(L_m, L_{n-1}) \rightarrow \cdots \\
\cdots \rightarrow \Hom(E^-_n, L_m) \rightarrow \Hom(L_n, L_m) \rightarrow \Ext^1(L_{n-1}, L_m) \rightarrow \Ext^1(E^-_n, L_m) \rightarrow \Ext^1(L_n, L_m) \rightarrow  \cdots
\end{split}
\end{equation}
we see that if $m \notin \{n, n-2, n+1, n-1\}$ then $\Ext^1(L_m, E^-_n) = 0$. We have $\Ext^1(L_m, E^-_n) = 0$ for $m=n, n-2$ by \eqref{van1} and \eqref{van2}. If $m=n-1$, then the part of the long-exact sequence becomes
\[
\cdots \rightarrow  0 \rightarrow \mathbb C \rightarrow \mathbb C \rightarrow \Ext^1(L_{n-1}, E^-_n) \rightarrow  0  \rightarrow \cdots
\]
and hence 
$\Ext^1(L_{n-1}, E^-_n) = 0$. if $m=n+1$, then it collapses to 
\[
\cdots \rightarrow  0 \rightarrow 0 \rightarrow \mathbb C \rightarrow \Ext^1(L_{n-1}, E^-_n) \rightarrow  0  \rightarrow \cdots
\]
and hence $\Ext^1(L_{n-1}, E^-_n) =\mathbb C$.

Next we observe that $\Ext^1(E^-_n, L_m) = 0$ for $m \notin \{n-1, n+1, n, n-2 \}$. By \eqref{van1} and \eqref{van2} this extension also vanishes for $m= n\pm1$. For $m=n$ we also get zero, since
\[
\cdots \rightarrow  0 \rightarrow \mathbb C \rightarrow \mathbb C \rightarrow \Ext^1(E^-_n, L_n) \rightarrow  0  \rightarrow \cdots
\]
Finally the case $m=n-2$ gives $\Ext^1(E^-_n, L_{n-2}) = \mathbb C$, since
\[
\cdots \rightarrow  0 \rightarrow 0 \rightarrow \mathbb C \rightarrow \Ext^1(E^-_n, L_{n-2}) \rightarrow  0  \rightarrow \cdots
\]

\subsection{The cases $\Ext^1(E^-_n, E^+_m)$ and $\Ext^1(E^+_n, E^-_m)$}

${}$ \newline 
Consider \eqref{long} for {$X = E^\pm_n$ and $\ses{L_m}{E^\mp_m}{L_{m \mp 1}}$ and}
$X = E^\pm_m$ and $\ses{L_n}{E^\mp_n}{L_{n \mp 1}}$
\begin{equation}
\begin{split}
\cdots \rightarrow \Ext^1(E^+_n,  L_{m}) \rightarrow \Ext^1(E^+_n, E^-_m) \rightarrow \Ext^1(E^+_n, L_{m-1}) \rightarrow  \cdots \\
\cdots  \rightarrow \Ext^1(E^-_n,  L_{m}) \rightarrow \Ext^1(E^-_n, E^+_m) \rightarrow \Ext^1(E^-_n, L_{m+1}) \rightarrow  \cdots \\
\cdots  \rightarrow \Ext^1(L_{n-1}, E^+_m) \rightarrow \Ext^1(E^-_n, E^+_m) \rightarrow \Ext^1(L_n, E^+_m) \rightarrow  \cdots \\
\cdots  \rightarrow \Ext^1(L_{n+1}, E^-_m) \rightarrow \Ext^1(E^+_n, E^-_m) \rightarrow \Ext^1(L_n, E^-_m) \rightarrow  \cdots
\end{split}
\end{equation} 
From the second sequence we see that $\Ext^1(E^-_n, E^+_m)= 0$ for $m \not\in\{ n-2, n-3\}$ and from the third one we see that $\Ext^1(E^-_n, E^+_m)= 0$ for $m \not\in\{ n, n+1\}$. Hence $\Ext^1(E^-_n, E^+_m) = 0$.  

From the first sequence we see that $\Ext^1(E^+_n, E^-_m)= 0$ for $m \not\in\{ n+2, n+3\}$ and from the fourth one we see that $\Ext^1(E^{{+}}_n, E^{{-}}_m)= 0$ for $m \not\in\{ n, n-1\}$. Hence $\Ext^1(E^+_n, E^-_m) = 0$ as well.  

\subsection{The cases $\Ext^1(E^+_n, E^+_m)$ and $\Ext^1(E^-_n, E^-_m)$}

${}$ \newline 
Consider \eqref{long} for $X = E^\pm_{{n}}$ and $\ses{L_{{m}}}{E^\pm_{{m}}}{L_{{{m}} \pm 1}}$
\begin{equation}
\begin{split}
\cdots \rightarrow \Hom^1(E^-_n, E^-_m) \rightarrow \Hom(E^-_n, L_{m-1}) \rightarrow \Ext^1(E^-_n,  L_{m}) \rightarrow \Ext^1(E^-_n, E^-_m) \rightarrow \Ext^1(E^-_n, L_{m-1}) \rightarrow  \cdots \\
\cdots  \rightarrow \Hom(E^+_n, E^+_m) \rightarrow \Hom(E^+_n, L_{m+1}) \rightarrow  \Ext^1(E^+_n,  L_{m}) \rightarrow \Ext^1(E^+_n, E^+_m) \rightarrow \Ext^1(E^+_n, L_{m+1}) \rightarrow  \cdots 
\end{split}
\end{equation} 
From the second sequence we see that $\Ext^1(E^+_n, E^+_m)= 0$ for $m \not\in\{ n+1, n+2\}$. 
$\Ext^1(E^+_n, E^+_{n+2})= \mathbb C$ follows immediately, while we can only deduce from this sequence that $\Ext^1(E^+_n, E^+_{n+1}) \in \{ 0, \mathbb C\}$, but zero is excluded by \eqref{ada}.
In complete analogy we get from the first sequence that $\Ext^1(E^-_n, E^-_m)= 0$ or $m \not\in\{ n-1, n-2\}$ and $\Ext^1(E^-_n, E^-_{n-1})= \mathbb C = \Ext^1(E^-_n, E^-_{n-2})$.
Define $P^\pm_n$ by {non-split exact sequences}
\begin{equation}
\begin{split}
\ses{E^+_n}{P^+_n}{E^+_{n-1}}\qquad  \in  \ \Ext^1(E^+_{n-1}, E^+_n) \\
\ses{E^-_n}{P^-_n}{E^-_{n+1}}\qquad  \in  \ \Ext^1(E^-_{n+1}, E^-_n) \\
\end{split}
\end{equation} 
We will see in a moment that $P^+_n \cong P^-_n$. 

\subsection{The cases $\Ext^1(P_n, L_m)$ and $\Ext^1(L_m, P_n)$ }

${}$ \newline
Consider \eqref{long} for $\ses{E^\pm_n}{P^\pm_n}{E^\pm_{n\mp 1}}$
\begin{equation}\nonumber
\begin{split}
 0 \rightarrow \Hom(X, E^+_n) \rightarrow \Hom(X, P^+_n) \rightarrow \Hom(X, E^+_{n-1}) \rightarrow \Ext^1(X, E^+_n) \rightarrow \Ext^1(X, P^+_n) \rightarrow \Ext^1(X, E^+_{n-1}) \rightarrow  \cdots \\
 0 \rightarrow \Hom(E^+_{n-1}, X) \rightarrow \Hom(P^+_n, X) \rightarrow \Hom(E^+_n, X) \rightarrow \Ext^1(E^+_{n-1}, X) \rightarrow \Ext^1(P^+_n, X) \rightarrow \Ext^1(E^+_n, X) \rightarrow  \cdots \\
 0 \rightarrow \Hom(X, E^-_n) \rightarrow \Hom(X, P^-_n) \rightarrow \Hom(X, E^-_{n+1}) \rightarrow \Ext^1(X, E^-_n) \rightarrow \Ext^1(X, P^-_n) \rightarrow \Ext^1(X, E^-_{n+1}) \rightarrow  \cdots \\
 0 \rightarrow \Hom(E^-_{n+1}, X) \rightarrow \Hom(P^-_n, X) \rightarrow \Hom(E^-_n, X) \rightarrow \Ext^1(E^-_{n+1}, X) \rightarrow \Ext^1(P^-_n, X) \rightarrow \Ext^1(E^-_n, X) \rightarrow  \cdots
\end{split}
\end{equation}
The first sequence with $X = E^-_n$ tells us that $\Hom(E^-_n, P^+_n) =\mathbb C$, the second sequence with $X = E^-_{n+1}$ tells us that $\Hom(P^+_n, E^-_{n+1}) =\mathbb C$. 

We show that a non-zero morphism $\phi:E^-_n\ra P^+_n$ is a monomorphism.
Suppose that it was not injective on the contrary. Then
$\im(\phi)\cong L_{n-1}$ and the module $P_n^+/\im(\phi)$ is an element of
$\Ext^1(L_n,E_n^+)$.
Since $\Ext^1(L_n,E_n^+)=0$ by \eqref{van2}, it follows that
the sequence $0\ra E_n^+\ra P_n^+/\im(\phi)\ra L_n\ra0$ splits.
Then we see that $E_{n-1}^+$ is a submodule of $P_n^+$, which contradicts that $P_n^+$ is a non-split extension.
Thus, we have $E^-_n\subset P^+_n$.

In a similar way, we see that there is an epimorphism $P^+_n\ra E^-_{n+1}$. Combining with the embedding $E^-_n\ra P^+_n$, we have 
an exact sequence
$$
0\ra E^-_n\ra P^+_n\ra E_{n+1}^-\ra 0.
$$
Since $P^+_n$ is indecomposable we get from this that $P^+_n \cong P^-_n$. Set $P_n := P^+_n$. 

Consider the first sequence with $X =L_m$, then $\Ext^1(L_m, P_n) = 0$ if $m \notin \{n-1, n-2\}$. But the third sequence with $X =L_m$ tells us that $\Ext^1(L_m, P_n) = 0$ if $m \notin \{n+1, n+2\}$, so that $\Ext^1(L_m, P_n) = 0$.

Consider the second sequence with $X =L_m$, then $\Ext^1(P^n, L_m) = 0$ if $m \notin \{n+1, n+2\}$. But the fourth sequence with $X =L_m$ tells us that $\Ext^1(P_n, L_m) = 0$ if $m \notin \{n-1, n-2\}$, so that $\Ext^1(P_n, L_m) = 0$.

\subsection{Loewy diagrams}

The Loewy diagram of the projective module $P_n$ and the typical modules $E^\pm_n$ are
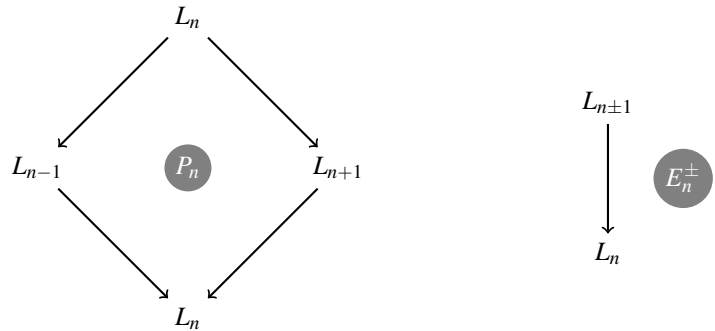
\begin{center}
\begin{tikzpicture}[scale=1]
\node (top) at (0,2) [] {$L_n$};
\node (left) at (-2,0) [] {$L_{n-1}$};
\node (right) at (2,0) [] {$L_{n+1}$};
\node (bottom) at (0,-2) [] {$L_n$};
\draw[->, thick] (top) -- (left);
\draw[->, thick] (top) -- (right);
\draw[->, thick] (left) -- (bottom);
\draw[->, thick] (right) -- (bottom);
\node (label) at (0,0) [circle, inner sep=2pt, color=white, fill=black!50!] {$P_n$};
\end{tikzpicture}
\hspace{2.5cm}
\begin{tikzpicture}[scale=1]
\node (top) at (0,1) [] {$L_{n \pm 1}$};
\node (right) at (0,-2) [] {${}$};
\node (bottom) at (0,-1) [] {$L_{n}$};
\draw[->, thick] (top) -- (bottom);
\node (label) at (1,0) [circle, inner sep=2pt, color=white, fill=black!50!] {$E^\pm_n$};
\end{tikzpicture}
\end{center}
$P_n$ is the projective cover and injective hull of $L_n$. Let $V_{n-1}^2$ denote the kernel of $P_n \rightarrow L_n$ and 
$\Lambda_{n-1}^2$ the quotient of $P_n$ by $L_n$, i.e. they are characterized by the short exact sequences
\[
\ses{V_{n-1}^2}{P_n}{L_n}\qquad \text{and} \qquad \ses{L_n}{P_n}{\Lambda_{n-1}^2}.
\]
and their Loewy diagrams are
\begin{center}
\begin{tikzpicture}[scale=1]
\node (top) at (0,2) [] {$L_n$};
\node (left) at (-2,0) [] {$L_{n-1}$};
\node (right) at (2,0) [] {$L_{n+1}$};
\draw[->, thick] (top) -- (left);
\draw[->, thick] (top) -- (right);
\node (label) at (0,0) [circle, inner sep=2pt, color=white, fill=black!50!] {$\Lambda_{n-1}^2$};
\end{tikzpicture}
\hspace{1.5cm}
\begin{tikzpicture}[scale=1]
\node (left) at (-2,0) [] {$L_{n-1}$};
\node (right) at (2,0) [] {$L_{n+1}$};
\node (bottom) at (0,-2) [] {$L_n$};
\draw[->, thick] (left) -- (bottom);
\draw[->, thick] (right) -- (bottom);
\node (label) at (0,0) [circle, inner sep=2pt, color=white, fill=black!50!] {$V_{n-1}^2$};
\end{tikzpicture}
\end{center}
Due to the shape of their Loewy diagrams we call these modules Zig-Zag modules. These modules can be extended by $L_n$ to $P_n$. In addition one finds that there exist also extensions by $L_{n-2}, L_n, L_{n+2}$ of the forms
\begin{equation}
\begin{split}
&\ses{L_{n-2}}{\Lambda_{n-2}^3}{V_{n-1}^2} \\
&\ses{L_n}{E^+_n \oplus E^-_n}{V_{n-1}^2} \\
&\ses{L_{n+2}}{V_{n-1}^3}{V_{n-1}^2} \\
&\ses{\Lambda_{n-1}^2 }{V_{n-2}^3}{L_{n-2}} \\
&\ses{\Lambda_{n-1}^2 }{E^+_{n-1} \oplus E^-_{n+1}}{L_n} \\
&\ses{\Lambda_{n-1}^2 }{\Lambda_{n-1}^3}{L_{n+2}} \\
\end{split}
\end{equation}
These modules $\Lambda_{n}^3, V_n^3$ are new indecomposables with decomposition factors our previously constructed indecomposables. 
One can now iterate this procedure and fortunately this has been done in section 3 of \cite{BRS} for  Temperly-Lieb algebra categories that are very similar to our category.

\subsection{Zig-Zag modules}\label{sec:zigzag}

It remains to understand extensions of indecomposable modules. That problem is exactly the same as the discussion of section 3 of \cite{BRS}.
Section 3.1, in particular Proposition 3.1 of \cite{BRS}  defines recursively modules $\Lambda_{n}^m, V_n^m$ for $m>2$. 
First one shows that $\Ext^1(L_{n+2m+1}, \Lambda_n^{2m}) = \mathbb C$ and defines $\Lambda_n^{2m+1}$ by the corresponding extension, that is
\[
\ses{\Lambda_n^{2m}}{\Lambda_n^{2m+1}}{L_{n+2m+1}}.
\]
Similarly one shows that $\Ext^1(\Lambda_{n}^{2m+1}, L_{n+2m+2}) = \mathbb C$
and defines $\Lambda_n^{2m+2}$ by the corresponding extension as well, that is
\[
\ses{L_{n+2m+2}}{\Lambda_n^{2m+2}}{\Lambda_n^{2m+1}}.
\]
The $V_n^m$ are constructed very similarly. First one shows that $\Ext^1(V_n^{2m}, L_{n_2m+1}) = \mathbb C$ and defines $V_n^{2m+1}$ by the corresponding extension, that is
\[
\ses{L_{n+2m+1}}{V_n^{2m+1}}{V_n^{2m}}.
\]
Similarly one shows that $\Ext^1(L_{n+2m+2}, V_{n}^{2m+1}) = \mathbb C$
and defines $V_n^{2m+2}$ by the corresponding extension as well, that is
\[
\ses{V_n^{2m+1}}{V_n^{2m+2}}{L_{n+2m+2}}.
\]
The Loewy diagrams of these modules are of the Zig-Zag form
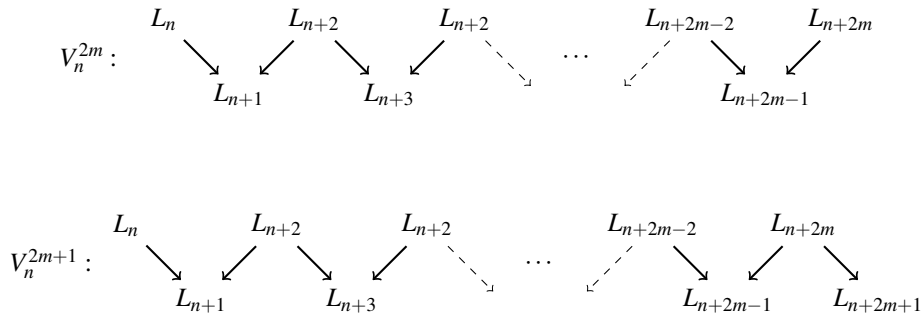
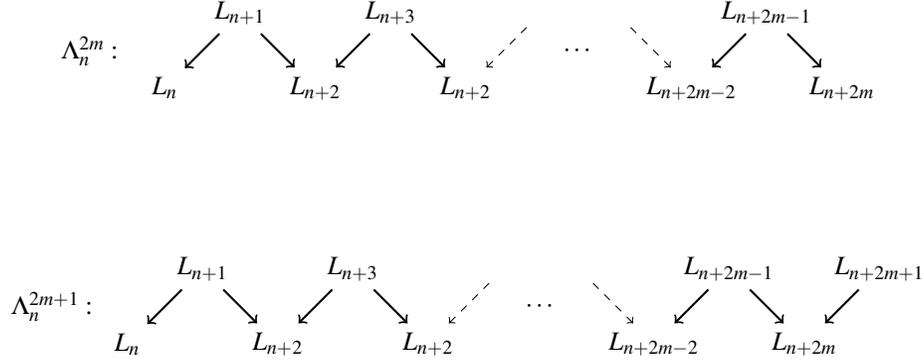
\begin{center}
\begin{tikzpicture}[scale=1]
\node (top1) at (1,1) [] {$L_n$};
\node (bot1) at (2,0) [] {$L_{n+1}$};
\node (top2) at (3,1) [] {$L_{n+2}$};
\node (bot2) at (4,0) [] {$L_{n+3}$};
\node (top3) at (5,1) [] {$L_{n+2}$};
\node(bot3) at (6,0) [] {$$};
\node(mid) at (6.5,0.5) [] {$\cdots$};
\node (top4) at (10,1) [] {$L_{n+2m}$};
\node (bot4) at (9,0) [] {$L_{n+2m-1}$};
\node (top5) at (8,1) [] {$L_{n+2m-2}$};
\node(bot5) at (7,0) [] {$$};
\draw[->, thick] (top1) -- (bot1);
\draw[->, thick] (top2) -- (bot1);
\draw[->, thick] (top2) -- (bot2);
\draw[->, thick] (top3) -- (bot2);
\draw[->, dashed] (top3) -- (bot3);
\draw[->, thick] (top4) -- (bot4);
\draw[->, thick] (top5) -- (bot4);
\draw[->, dashed] (top5) -- (bot5);
\node (label) at (0,0.5) [] {$V_{n}^{2m}:$};
\end{tikzpicture}
\end{center}
\vspace{1cm}
\begin{center}
\begin{tikzpicture}[scale=1]
\node (top1) at (1,1) [] {$L_n$};
\node (bot1) at (2,0) [] {$L_{n+1}$};
\node (top2) at (3,1) [] {$L_{n+2}$};
\node (bot2) at (4,0) [] {$L_{n+3}$};
\node (top3) at (5,1) [] {$L_{n+2}$};
\node(bot3) at (6,0) [] {$$};
\node(mid) at (6.5,0.5) [] {$\cdots$};
\node (botextra) at (11,0) [] {$L_{n+2m+1}$};
\node (top4) at (10,1) [] {$L_{n+2m}$};
\node (bot4) at (9,0) [] {$L_{n+2m-1}$};
\node (top5) at (8,1) [] {$L_{n+2m-2}$};
\node(bot5) at (7,0) [] {$$};
\draw[->, thick] (top1) -- (bot1);
\draw[->, thick] (top2) -- (bot1);
\draw[->, thick] (top2) -- (bot2);
\draw[->, thick] (top3) -- (bot2);
\draw[->, dashed] (top3) -- (bot3);
\draw[->, thick] (top4) -- (bot4);
\draw[->, thick] (top5) -- (bot4);
\draw[->, dashed] (top5) -- (bot5);
\draw[->, thick] (top4) -- (botextra);
\node (label) at (0,0.5) [] {$V_{n}^{2m+1}:$};
\end{tikzpicture}
\end{center}
and
\begin{center}
\begin{tikzpicture}[scale=1]
\node (top1) at (1,0) [] {$L_n$};
\node (bot1) at (2,1) [] {$L_{n+1}$};
\node (top2) at (3,0) [] {$L_{n+2}$};
\node (bot2) at (4,1) [] {$L_{n+3}$};
\node (top3) at (5,0) [] {$L_{n+2}$};
\node(bot3) at (6,1) [] {$$};
\node(mid) at (6.5,0.5) [] {$\cdots$};
\node (top4) at (10,0) [] {$L_{n+2m}$};
\node (bot4) at (9,1) [] {$L_{n+2m-1}$};
\node (top5) at (8,0) [] {$L_{n+2m-2}$};
\node(bot5) at (7,1) [] {$$};
\draw[<-, thick] (top1) -- (bot1);
\draw[<-, thick] (top2) -- (bot1);
\draw[<-, thick] (top2) -- (bot2);
\draw[<-, thick] (top3) -- (bot2);
\draw[<-, dashed] (top3) -- (bot3);
\draw[<-, thick] (top4) -- (bot4);
\draw[<-, thick] (top5) -- (bot4);
\draw[<-, dashed] (top5) -- (bot5);
\node (label) at (0,0.5) [] {$\Lambda_{n}^{2m}:$};
\end{tikzpicture}
\end{center}
\vspace{1cm}
\begin{center}
\begin{tikzpicture}[scale=1]
\node (top1) at (1,0) [] {$L_n$};
\node (bot1) at (2,1) [] {$L_{n+1}$};
\node (top2) at (3,0) [] {$L_{n+2}$};
\node (bot2) at (4,1) [] {$L_{n+3}$};
\node (top3) at (5,0) [] {$L_{n+2}$};
\node(bot3) at (6,1) [] {$$};
\node(mid) at (6.5,0.5) [] {$\cdots$};
\node (botextra) at (11,1) [] {$L_{n+2m+1}$};
\node (top4) at (10,0) [] {$L_{n+2m}$};
\node (bot4) at (9,1) [] {$L_{n+2m-1}$};
\node (top5) at (8,0) [] {$L_{n+2m-2}$};
\node(bot5) at (7,1) [] {$$};
\draw[<-, thick] (top1) -- (bot1);
\draw[<-, thick] (top2) -- (bot1);
\draw[<-, thick] (top2) -- (bot2);
\draw[<-, thick] (top3) -- (bot2);
\draw[<-, dashed] (top3) -- (bot3);
\draw[<-, thick] (top4) -- (bot4);
\draw[<-, thick] (top5) -- (bot4);
\draw[<-, dashed] (top5) -- (bot5);
\draw[<-, thick] (top4) -- (botextra);
\node (label) at (0,0.5) [] {$\Lambda_{n}^{2m+1}:$};
\end{tikzpicture}
\end{center}

It remains to show that these modules are closed under extensions. 
Proposition 3.8 of \cite{BRS} gives for $m\neq 0$, that 
\begin{equation}\label{list}
\begin{split}
\Ext^1(L_r, \Lambda_n^{2m}) &= \bigoplus_{s=0}^{m+1} \delta_{r, n+2s-1} \mathbb C = \Ext^1(V_n^{2m}, L_r) \\
\Ext^1(L_r, \Lambda_n^{2m+1}) &= \bigoplus_{s=0}^{m} \delta_{r, n+2s-1} \mathbb C = \Ext^1(V_n^{2m+1}, L_r) \\
\Ext^1(\Lambda_n^{2m}, L_r) &= \delta_{m, 1} \delta_{r, n+1}\mathbb C \oplus  \bigoplus_{s=1}^{m-1} \delta_{r, n+2s} \mathbb C = \Ext^1(L_r, V_n^{2m}) \\
\Ext^1(\Lambda_n^{2m+1}, L_r) &= \bigoplus_{s=1}^{m+1} \delta_{r, n+2s} \mathbb C = \Ext^1(L_r, V_n^{2m+1}).
\end{split}
\end{equation}
$\Ext^1(\Lambda_n^{2}, L_{n+1})$ and $ \Ext^1(L_{n+1}, V_n^{2})$ were discussed in the previous section and the extension is the projective module $P_{n+1}$. 
The construction of all other extensions is the same as Proposition 3.9 of \cite{BRS} and with the convenient notation $\Lambda_n^{-1} := V_n^{-1}:= 0$, $\Lambda_n^0 := V_n^0 := L_n$ and $\Lambda_n^1:= E^+_n, V_n^1 := E^-_{n+1}$ one gets 
\begin{equation}\label{eq5.15}
\begin{split}
&\ses{\Lambda_n^{2m}}{\Lambda_{n}^{2s-1} \oplus V_{n+2s-1}^{2m-2s+1}}{L_{n+2s-1}}, \qquad \text{for} \ 0 \leq s \leq  m+1 \\ 
&\ses{L_{n+2s}}{\Lambda_{n}^{2s} \oplus \Lambda_{n+2s}^{2m-2s}}{\Lambda_n^{2m}}, \qquad \text{for} \ 1 \leq s \leq  m-1 \\ 
&\ses{\Lambda_n^{2m+1}}{\Lambda_{n}^{2s-1} \oplus V_{n+2s-1}^{2m-2s+2}}{L_{n+2s-1}}, \qquad \text{for} \ 0 \leq s \leq  m \\ 
&\ses{L_{n+2s}}{\Lambda_{n}^{2s} \oplus \Lambda_{n+2s}^{2m-2s+1}}{\Lambda_n^{2m+1}}, \qquad \text{for} \ 1 \leq s \leq  m+1 \\
&\ses{ L_{n+2s-1}}{V_{n}^{2s-1} \oplus \Lambda_{n+2s-1}^{2m-2s+1}}{V_n^{2m}}, \qquad \text{for} \ 0 \leq s \leq  m+1 \\ 
&\ses{V_n^{2m} }{V_{n}^{2s} \oplus V_{n+2s}^{2m-2s}}{L_{n+2s}}, \qquad \text{for} \ 1 \leq s \leq  m-1 \\ 
&\ses{ L_{n+2s-1}}{V_{n}^{2s-1} \oplus \Lambda_{n+2s-1}^{2m-2s+2}}{V_n^{2m+1}}, \qquad \text{for} \ 0 \leq s \leq  m \\ 
&\ses{V_n^{2m+1} }{V_{n}^{2s} \oplus V_{n+2s}^{2m-2s+1}}{L_{n+2s}}, \qquad \text{for} \ 1 \leq s \leq  m+1 \\ 
\end{split}
\end{equation}
We claim that any non-projective indecomposable object of length $m$ is either isomorphic to $\Lambda^{m-1}_n$ or $V^{m-1}_n$ for some $n$. This follows by induction for $m$, since any indecomposable object of length $m$ is an extension by a simple object of length $m-1$.  With the exception of $\Ext^1(\Lambda_n^{2}, L_{n+1})$ and $ \Ext^1(L_{n+1}, V_n^{2})$ that lead to projective indecomposables any extension of  $\Lambda^{m-2}_n$ or $V^{m-2}_n$ appears in \eqref{list} and if they are indecomposable then they are
one of the four  $\Lambda^{m-1}_n, V^{m-1}_n, \Lambda^{m-1}_{n-1}$ or  $V^{m-1}_{n-1}$. 

The projective cover of these modules is the projective cover of its head and the injective hull is the one of its socle. 
Set 
\[
R_n^{m} := \bigoplus_{s = 0}^m P_{n+2s},
\]
then the Zig-Zag modules satisfy
\begin{equation}
\begin{split}
&\ses{V_{n}^{2m}}{R_n^{m}}{V_{n-1}^{m+1}} \\
&\ses{\Lambda_{n-1}^{m+1}}{R_n^m}{\Lambda_n^m}.  
\end{split}
\end{equation}
Splicing these one gets the injective and projective resolutions of simple objects
\begin{equation}
\begin{split}
0 \rightarrow L_n \rightarrow R_n^0 \rightarrow R_{n-1}^1  \rightarrow R_{n-2}^2  \rightarrow R_{n-3}^3 \rightarrow  \cdots \\ 
\cdots \rightarrow R_{n-3}^3  \rightarrow R_{n-2}^2 \rightarrow R_{n-1}^1 \rightarrow R_{n}^0 \rightarrow L_n  \rightarrow  0.
\end{split}
\end{equation}
The projective resolution yields the cochain complex
\[
0 \rightarrow \text{Hom}(R_n^0,  \bullet) \rightarrow \text{Hom}(R_{n-1}^1,  \bullet) \rightarrow \text{Hom}(R_{n-2}^2,  \bullet) \rightarrow \text{Hom}(R_{n-3}^3,  \bullet) \rightarrow \cdots
\]
We have 
\begin{equation}
\begin{split}
\text{Hom}(R_{n-s}^s, L_m) &= \text{Hom}(P_{n-s} \oplus P_{n-s+2} \oplus \dots \oplus P_{n+s-2} \oplus P_{n+2}, L_m) \\
 &= \text{Hom}(L_{n-s} \oplus L_{n-s+2} \oplus \dots \oplus L_{n+s-2} \oplus L_{n+2}, L_m) \\
 &= \begin{cases} \mathbb C & \ \text{if} \ | n- m | \leq  s \ \text{and} \ n -s = m \mod  2 \\ 0 & \ \text{else} \end{cases}
\end{split}
\end{equation}
In particular if $\text{Hom}(R_{n-s}^s, L_m) \neq 0$ then $\text{Hom}(R_{n-(s \pm 1)}^{s \pm 1}, L_m) =0$ and so the coboundary of this cochain complex is trivial. Thus 
\[
\Ext^s(L_n, L_m) = \text{Hom}(R_{n-s}^s, L_m) =  \begin{cases} \mathbb C & \ \text{if} \ | n- m | \leq  s \ \text{and} \ n -s  = m \mod  2 \\ 0 & \ \text{else} \end{cases}
\]

\appendix \section{Constructing the $P_n$}\label{appendix}

In this section let $V$ be a VOA containing a Heisenberg subVOA $H$ generated by a Heisenberg field $J$. Let $J_0$ be the zero-mode. 
This section uses the contragredient dual and so we recall Proposition 4.3.1 of \cite{CMY4}: Let $M, N, R$ be modules for a VOA $V$ and $M', N', R'$ be their contragredient duals. 
If 
\[
 M \xrightarrow{f} N \xrightarrow{g} R 
\]
is exact, then so is 
\[
R' \xrightarrow{g'} N' \xrightarrow{f'} M'. 
\]

We assume the following

\begin{assum}\label{assum_on_modules}
There are simple $V$-modules $A, B, C$ and a complex number $h$, such that $B \not\cong C$ and 
\begin{enumerate}
\item $J_0$ acts semisimply on $A, B, C$ and  the $J_0$-eigenvalues on the modules $A, B, C$ lie in $h +\mathbb Z$
\item $\Ext^1(A, A) = \Ext^1(B, B) = \Ext^1(C, C) = 0$. 
\item There exist modules $E_B, F_C$ satisfying the non-split short-exact  sequences
\[
\ses{B}{E_B}{A}\qquad \text{and} \qquad \ses{A}{F_C}{C}.
\]
\item Let $\pi = \text{exp}(-2\pi i h)\mathrm{Id} \ \text{exp}(2\pi i J_0) - \mathrm{Id}$. There exist modules $E_B^{(2)}, F_C^{(2)}$ satisfying the non-split short-exact sequences 
\[
0\rightarrow {E_B} \rightarrow  {E_B^{(2)}}\xrightarrow{\pi}{E_B}\rightarrow 0
\qquad \text{and} \qquad 
0\rightarrow {F_C} \rightarrow  {F_C^{(2)}}\xrightarrow{\pi}{F_C}\rightarrow 0
 \]
\end{enumerate}

\end{assum}

\begin{theorem}\label{thm:appendix}
Retain Assumption \ref{assum_on_modules}, then there exists an indecomposable module $P_A$ on which $J_0$ acts semisimply and that has Loewy diagram
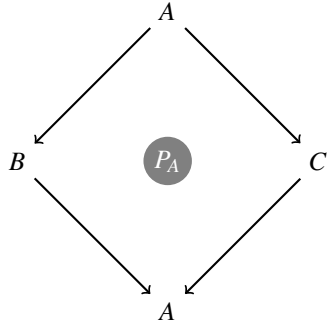
\begin{center}
\begin{tikzpicture}[scale=1]
\node (top) at (0,2) [] {$A$};
\node (left) at (-2,0) [] {$B$};
\node (right) at (2,0) [] {$C$};
\node (bottom) at (0,-2) [] {$A$};
\draw[->, thick] (top) -- (left);
\draw[->, thick] (top) -- (right);
\draw[->, thick] (left) -- (bottom);
\draw[->, thick] (right) -- (bottom);
\node (label) at (0,0) [circle, inner sep=2pt, color=white, fill=black!50!] {$P_A$};
\end{tikzpicture}
\end{center}

\end{theorem}
\begin{proof}
The proof is inspired by Appendix B of \cite{Garner:2023pmt}.
Since $\Ext^1(A, A) = \Ext^1(B, B) =\Ext^1(C, C) = 0$ the Loewy diagrams of $E_B^{(2)}, F_C^{(2)}$ can only be of the form

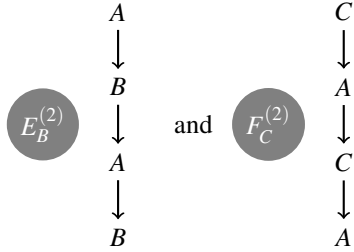
\begin{center}
\begin{tikzpicture}[scale=1]
\node (top) at (0,2) [] {$A$};
\node (left) at (0,1) [] {$B$};
\node (right) at (0,0) [] {$A$};
\node (bottom) at (0,-1) [] {$B$};
\draw[->, thick] (top) -- (left);
\draw[->, thick] (left) -- (right);
\draw[->, thick] (right) -- (bottom);
\node (label) at (-1,0.5) [circle, inner sep=2pt, color=white, fill=black!50!] {$E_B^{(2)}$};
\node at (1,0.5) []{and};
\node (top) at (3,2) [] {$C$};
\node (left) at (3,1) [] {$A$};
\node (right) at (3,0) [] {$C$};
\node (bottom) at (3,-1) [] {$A$};
\draw[->, thick] (top) -- (left);
\draw[->, thick] (left) -- (right);
\draw[->, thick] (right) -- (bottom);
\node (label) at (2,0.5) [circle, inner sep=2pt, color=white, fill=black!50!] {$F_C^{(2)}$};
\end{tikzpicture}
\end{center}
Let $G_B$ be the the cokernel of the embedding of $B$ in $E_B^{(2)}$ and let $G_C$ be the kernel of the projection of $F_C^{(2)}$ onto $C$. The resulting Loewy diagrams are
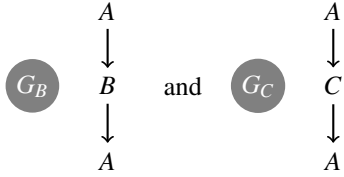
\begin{center}
\begin{tikzpicture}[scale=1]
\node (top) at (0,2) [] {$A$};
\node (left) at (0,1) [] {$B$};
\node (right) at (0,0) [] {$A$};
\node at (1,1) []{and};
\draw[->, thick] (top) -- (left);
\draw[->, thick] (left) -- (right);
\node (label) at (-1,1) [circle, inner sep=2pt, color=white, fill=black!50!] {$G_B$};
\node (top) at (3,2) [] {$A$};
\node (left) at (3,1) [] {$C$};
\node (right) at (3,0) [] {$A$};
\draw[->, thick] (top) -- (left);
\draw[->, thick] (left) -- (right);
\node (label) at (2,1) [circle, inner sep=2pt, color=white, fill=black!50!] {$G_C$};
\end{tikzpicture}
\end{center}
For both $G_B, G_C$ the image of $\pi$ is the module $A$, that is we have surjections $\pi_B: G_B \rightarrow A, \pi_C: G_C  \rightarrow A$. 
Let $G_B', G_C'$ be the contragredient dual of $G_B, G_C$ so that $\pi_B, \pi_C$ correspond to injections $\iota_B: A' \rightarrow G_B', \iota_C: A' \rightarrow G_C'$. 
Now, 
 $\pi' := \text{exp}(2\pi i h)\mathrm{Id} \ \text{exp}(2\pi i J_0) - \mathrm{Id}$ acts nilpotently on $G_B', G_C'$ with image $A'$, the contragredient dual of $A$. In particular we have surjective homomorphisms $\varphi_B : G_B' \rightarrow A'$ and  $\varphi_C : G_C' \rightarrow A'$ and hence one from the direct sum,  $\varphi: G_B'\oplus G_C'$ to $A'$. The restriction of $\varphi$ to each summand is $\varphi_B$ respectively $\varphi_C$. 
Denote by $K_A$ the kernel of $\varphi$, i.e.
\begin{equation}\label{eq:ses_KA}
\ses{K_A}{G_B'\oplus G_C'}{A'}.
\end{equation}
$K_A$ is indecomposable and  has Loewy diagram
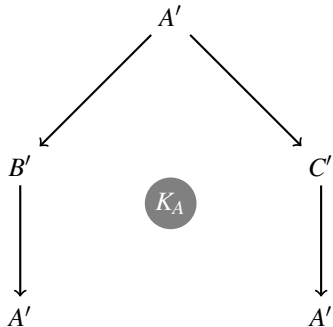
\begin{center}
\begin{tikzpicture}[scale=1]
\node (top) at (0,2) [] {$A'$};
\node (left) at (-2,0) [] {$B'$};
\node (right) at (2,0) [] {$C'$};
\node (bottomleft) at (-2,-2) [] {$A'$};
\node (bottomright) at (2,-2) [] {$A'$};
\draw[->, thick] (top) -- (left);
\draw[->, thick] (top) -- (right);
\draw[->, thick] (left) -- (bottomleft);
\draw[->, thick] (right) -- (bottomright);
\node (label) at (0,-0.5) [circle, inner sep=2pt, color=white, fill=black!50!] {$K_A$};
\end{tikzpicture}
\end{center}
Here $B', C'$ are the contragredient duals of $B, C$.
Taking the contragredient dual of \eqref{eq:ses_KA} gives
\[
\ses{A}{G_B \oplus G_C}{K_A'}
\]
with $K_A'$ the contragredient dual of $K_A$. Its Loewy diagram is 
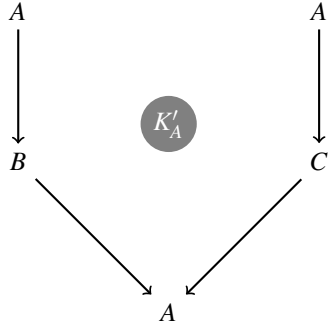
\begin{center}
\begin{tikzpicture}[scale=1]
\node (topleft) at (-2,2) [] {$A$};
\node (topright) at (2,2) [] {$A$};
\node (left) at (-2,0) [] {$B$};
\node (right) at (2,0) [] {$C$};
\node (bottom) at (0,-2) [] {$A$};
\draw[->, thick] (topleft) -- (left);
\draw[->, thick] (topright) -- (right);
\draw[->, thick] (left) -- (bottom);
\draw[->, thick] (right) -- (bottom);
\node (label) at (0,0.5) [circle, inner sep=2pt, color=white, fill=black!50!] {$K_A'$};
\end{tikzpicture}
\end{center}
The injections $A' \xrightarrow{\iota_B} G_B'  \rightarrow G_B' \oplus G_C',  A' \xrightarrow{\iota_C} G_C' \rightarrow G_B' \oplus G_C''$ factor through $K_A$ and so dually the projections  $G_B  \oplus G_C \rightarrow G_B \xrightarrow{\pi_B} A, G_B  \oplus G_C \rightarrow G_C \xrightarrow{\pi_C} A$ factor through $K_A'$.  By construction $\pi_B + \pi_c$ is isomorphic to the image of $\pi$. Let $P_A$ be the kernel of $\pi : K_A \rightarrow A = \text{image}(\pi : K_A \rightarrow K_A)$. $P_A$ being in the kernel of $\pi$ means that $J_0$ acts semisimply on $P_A$. Moreover the only possibility for the Loewy diagram of $P_A$ is 
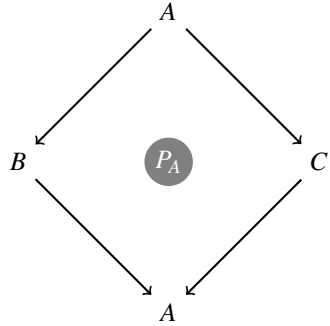
\begin{center}
\begin{tikzpicture}[scale=1]
\node (top) at (0,2) [] {$A$};
\node (left) at (-2,0) [] {$B$};
\node (right) at (2,0) [] {$C$};
\node (bottom) at (0,-2) [] {$A$};
\draw[->, thick] (top) -- (left);
\draw[->, thick] (top) -- (right);
\draw[->, thick] (left) -- (bottom);
\draw[->, thick] (right) -- (bottom);
\node (label) at (0,0) [circle, inner sep=2pt, color=white, fill=black!50!] {$P_A$};
\end{tikzpicture}
\end{center}
as desired. 
\end{proof}

\subsection{Proof of Proposition \ref{prop:log}}

We need to show that there exist indecomposable modules $\sfmod{\ell}{\slproj{r, s}^\pm}$ satisfying the non-split short exact sequences
\begin{align*}
&\dses {\sfmod{\ell}{\slindrel{r,s}^+}} {\sfmod{\ell}{\slproj{r, s}^+}}    {\sfmod{\ell+1}{\slindrel{r,s+1}^+}},\\
&\dses {\sfmod{\ell}{\slindrel{r,v-1}^+}} {\sfmod{\ell}{\slproj{r, v-1}^+}}    {\sfmod{\ell+2}{\slindrel{u-r,1}^+}},\\
&\dses {\sfmod{\ell}{\slindrel{r,s}^-}} {\sfmod{\ell}{\slproj{r, s}^-}}    {\sfmod{\ell-1}{\slindrel{r,s+1}^-}} ,\\
&\dses {\sfmod{\ell}{\slindrel{r,v-1}^-}} {\sfmod{\ell}{\slproj{r, v-1}^-}}    {\sfmod{\ell-2}{\slindrel{u-r,1}^-}} ,\\
\end{align*}
for $\ell \in\ZZ, r = 1, \dots,  u-1, s = 1, \dots, v-2$. 
We set 
\[
A:= \sfmod{\ell}{\sldis{r,s}^-}, \qquad B:= \begin{cases} \sfmod{\ell-1}{\sldis{r,s+1}^-} & s =1, \dots, v-2 \\ \sfmod{\ell-2}{\sldis{u-r,1}^-} & s = v-1 \end{cases}, \qquad
C:= \sfmod{\ell}{\sldis{u-r,v-s}^+}
\]
We need to verify Assumption \ref{assum_on_modules}:
\begin{enumerate}
\item Easily verified
\item Special case of Theorem \ref{thm:van_ext}
\item Holds with 
\[
E_B:= \begin{cases} \sfmod{\ell-1}{\slindrel{r,s+1}^-}& s =1, \dots, v-2 \\  \sfmod{\ell-2}{\slindrel{u-r,1}^-}& s = v-1 \end{cases}, \qquad
F_C:=  \sfmod{\ell}{\slindrel{r,s}^-}
\]
by Proposition \ref{existelle}.
\item Holds by Corollary 5.3 of \cite{C}
\end{enumerate}

The Proposition follows from Theorems \ref{thm:appendix} and \ref{thm:van_ext}.

\flushleft
\bibliographystyle{unsrt}

\end{document}